\documentclass{article}
\usepackage{geometry}
\usepackage{amsmath}
\usepackage{latexsym}
\usepackage{amsthm}
\usepackage{amsfonts}
\usepackage{amssymb}
\usepackage{tikz}
\usepackage{tikz-cd}
\usepackage{gensymb}
\usepackage{cancel}
\usepackage{setspace}
\usepackage{mathtools}
\usepackage{framed}
\usepackage{tgbonum}
\usepackage{hyperref}
\usepackage{graphicx}
\usetikzlibrary{decorations.pathreplacing,angles,quotes}
\date{}
\author{Adam Afandi}
\title{An Ehrhart Theory For Tautological Intersection Numbers}

\newtheorem{Definition}{Definition}
\newtheorem{Proposition}{Proposition}
\newtheorem{Lemma}{Lemma}
\newtheorem{Theorem}{Theorem}
\newtheorem{Corollary}{Corollary}
\newtheorem{Example}{Example}
\newtheorem{Remark}{Remark}
\newtheorem{Open Problem}{Open Problem}

\newtheorem{Question}{Question}

\newcommand{\mbar}{\overline{\mathcal{M}}}

\begin{document}
\maketitle

\begin{abstract}

We discover that tautological intersection numbers on $\mbar_{g, n}$, the moduli space of stable genus $g$ curves with $n$ marked points, are evaluations of Ehrhart polynomials of partial polytopal complexes. In order to prove this, we realize the Virasoro constraints for tautological intersection numbers as a recursion for integer-valued polynomials. Then we apply a theorem of Breuer that classifies Ehrhart polynomials of partial polytopal complexes by the nonnegativity of their $f^*$-vector. In dimensions 1 and 2, we show that the polytopal complexes that arise are \emph{inside-out polytopes} i.e. polytopes that are dissected by a hyperplane arrangement.

\end{abstract}

{\small\tableofcontents}

\section{Introduction}
The purpose of this paper is to present a novel perspective concerning tautological intersection numbers on $\mbar_{g, n}$, the moduli space of stable $n$-pointed genus $g$ curves. This perspective uses ideas from a (seemingly) distant subfield of mathematics, namely, \emph{Ehrhart theory}. 

The main idea behind the present paper can be summarized as follows: \emph{tautological intersection numbers can be organized into evaluations of Ehrhart polynomials of partial polytopal complexes}.

Our intent is to present an exposition that is accessible to both algebraic geometers interested in $\mbar_{g, n}$, and combinatorialists and discrete geometers coming from Ehrhart theory. 

\subsection{Main Result}

An important algebraic object attached to the moduli space of pointed stable curves is the \emph{tautological ring}:

\begin{equation*}
R^*(\mbar_{g, n})
\end{equation*}

\noindent This ring is a subring of $A^*(\mbar_{g, n})$, the Chow ring of $\mbar_{g, n}$. Beginning with the work of Mumford \cite{Mumford83}, great strides have been made in our understanding of $R^*(\mbar_{g, n})$. In particular, many important cycles in $A^*(\mbar_{g, n})$ have been shown to be tautological. \\

\noindent Let $\alpha \in R^{3g - 3 + n}(\mbar_{g, n})$. Since $\text{dim}(\mbar_{g, n}) = 3g - 3 + n$, we can integrate $\alpha$ against the fundamental class of $\mbar_{g, n}$ to obtain a \emph{tautological intersection number}:

\begin{equation*}
\left(\int_{\mbar_{g, n}}\alpha\right) \in \mathbb{Q}
\end{equation*}

\noindent Let $\mathbb{L}_i$ be the $i^{th}$ universal cotangent line bundle on $\mbar_{g, n}$ i.e. the line bundle whose fiber over a point $[C, p_1, \ldots, p_n] \in \mbar_{g, n}$ is $T_{p_i}^*C$, the cotangent space to the $i^{th}$ marked point. Define $\psi_i$ to be the first Chern class of $\mathbb{L}_i$,

\begin{equation*}
\psi_i := c_1(\mathbb{L}_i)
\end{equation*}

\noindent These elements in $R^1(\mbar_{g, n})$ are usually referred to as $\psi$-classes. They play a central role in the tautological intersection theory of $\mbar_{g, n}$ for a multitude of reasons. In particular, all tautological intersection numbers can be reduced to intersection numbers only involving $\psi$-classes, that is, intersection numbers of the form

\begin{equation*}
\left<\tau_{d_1}\ldots\tau_{d_n}\right>_g := \int_{\mbar_{g, n}}\psi_1^{d_1}\ldots\psi_n^{d_n}
\end{equation*}

\noindent The main theorem of this paper shows that these intersection numbers are evaluations of \emph{Ehrhart polynomials} of partial polytopal complexes. An Ehrhart polynomial is a counting function for integer lattice points of dilations of an integral polytope. More precisely, given an integral $d$-polytope $P \subset \mathbb{R}^d$, its Ehrhart polynomial is

\begin{equation*}
L_P(g) := \left| gP \cap \mathbb{Z}^d \right|
\end{equation*}

\noindent where $g$ is a positive integer, and $gP$ is the $g^{th}$ dilate of $P$. One can extend the notion of an Ehrhart polynomial to more general polyhedral objects, in particular, to \emph{partial polytopal complexes}, which are disjoint unions of open polytopes (see Section \ref{Ehrhart} below). \\

\noindent Here is our main theorem:

\begin{Theorem}\label{MainTheorem}
Let $n \geq 1$, $\vec{d} := (d_1, \ldots, d_n) \in \mathbb{Z}^n_{\geq 0}$. Define $|\vec{d}| := \sum_{i = 1}^n d_i$, $C(\vec{d}) := \prod_{i = 1}^n (2d_i+ 1)!!$, and $m(\vec{d}) = m := \left\lceil \frac{2 - n + |\vec{d}|}{3} \right\rceil - 1$. Then there exists an integral partial polytopal complex $P_{\vec{d}}$ of dimension $|\vec{d}|$ and volume $\text{vol}\left(P_{\vec{d}}\right) = 6^{|\vec{d}|}$ such that

\begin{equation*}
24^{g + m}(g + m)!C(\vec{d})\int_{\mbar_{g + m, n + 1}}\psi_1^{d_1}\ldots\psi_n^{d_n}\psi_{n + 1}^{3(g + m) - 2 + n - |\vec{d}|} = \#\{\text{integer lattice points in $gP_{\vec{d}}$}\}
\end{equation*}

\noindent where $gP_{\vec{d}}$ is the $g^{th}$ dilate of $P_{\vec{d}}$.  

\end{Theorem}

\noindent The statement of Theorem \ref{MainTheorem} has some idiosyncratic notation, and might seem a bit opaque, so let us take a moment to explain how one should think about what Theorem \ref{MainTheorem} actually says. \\

\noindent Suppose you fix an integer vector $\vec{d} = (d_1, \ldots, d_n) \in \mathbb{Z}_{\geq 0}^n$ which corresponds to a monomial of $\psi$-classes $\psi_1^{d_1}\ldots\psi_n^{d_n}$. Consider the family of intersection numbers

\begin{equation*}
\left\{ \left<\tau_{d_1}\ldots\tau_{d_n}\tau_{d_{n + 1}}\right>_g \right\}_{g. d_{n + 1} \geq 0}
\end{equation*}

\noindent Since $\vec{d}$ is fixed, in order for $\left<\tau_{d_1}\ldots\tau_{d_n}\tau_{d_{n+1}}\right>_g$ to be nonzero, $d_{n + 1}$ must be 

\begin{equation*}
3g - 3 + (n + 1) - |\vec{d}| = 3g - 2 + n - |\vec{d}|
\end{equation*}

\noindent which explains the exponent of the last insertion in Theorem \ref{MainTheorem}. Furthermore, notice that there exists a smallest genus $g$ such that $\left<\tau_{d_1}\ldots\tau_{d_n}\tau_{d_{n+1}}\right>_g \not= 0$. This genus is the smallest genus $g$ such that the exponent on the last insertion is nonnegative, that is, the smallest genus $g$ such that

\begin{equation*}
3g - 2 + n - |\vec{d}| \geq 0
\end{equation*}

\noindent which is precisely

\begin{equation*}
\left\lceil \frac{2 - n + |\vec{d}|}{3} \right\rceil
\end{equation*}

\noindent Consequently, we see that $m(\vec{d}) := \left\lceil \frac{2 - n + |\vec{d}|}{3} \right\rceil - 1$ is designed to be an appropriate shift of the genera, in that it ensures the following equivalence:

\begin{equation*}
\left<\tau_{d_1}\ldots\tau_{d_n}\tau_{d_{n+1}}\right>_{g + m(\vec{d})} \not = 0 \iff g \geq 1
\end{equation*}

\noindent The statement of Theorem 1 then says that there exists a partial polytopal complex $P_{\vec{d}}$ that only depends on $\vec{d}$, such that

\begin{equation}\label{firstinterpretation}
24^{g + m(\vec{d})}(g + m(\vec{d}))!C(\vec{d})\left<\tau_{d_1}\ldots\tau_{d_n}\tau_{d_{n+1}}\right>_{g + m(\vec{d})} = \#\{\text{integer lattice points in $gP_{\vec{d}}$}\}
\end{equation}

\noindent Phrased in this way, we see that the smallest genus in which $\left<\tau_{d_1}\ldots\tau_{d_n}\tau_{d_{n+1}}\right>_{g + m(\vec{d})} \not= 0$ corresponds to the first dilate of $P_{\vec{d}}$, the next smallest genus will correspond to the 2nd dilate of $P_{\vec{d}}$, and so on. Of course, one could easily rewrite the equation in Theorem \ref{MainTheorem} as

\begin{equation}\label{secondinterpretation}
24^gg!C(\vec{d})\left<\tau_{d_1}\ldots\tau_{d_n}\tau_{d_{n+1}}\right>_g = \#\{\text{integer lattices points in $(g - m(\vec{d}))P_{\vec{d}}$}\}
\end{equation}

\noindent However, notice that Equation \ref{secondinterpretation} is only valid for $g \geq \left\lceil \frac{2 - n + |\vec{d}|}{3}\right\rceil$. Both ways of phrasing Theorem \ref{MainTheorem} i.e. Equation \ref{firstinterpretation} and Equation \ref{secondinterpretation}, are equivalent. However, Equation \ref{firstinterpretation} emphasizes the role of the partial polytopal complex and its dilates, while Equation \ref{secondinterpretation} emphasizes the role of the intersection numbers $\left< \tau_{d_1}\ldots\tau_{d_n}\tau_{d_{n+1}}\right>_g$. In the present paper, we have chosen the former.  \\

\noindent In addition to the connection with lattice-point counting in partial polytopal complexes, Theorem \ref{MainTheorem} implies that, as a formal consequence of Ehrhart theory, the quantity

\begin{equation*}
24^gg!\left<\tau_{d_1}\ldots\tau_{d_n}\tau_{3g - 2 + n - |\vec{d}|}\right>_g
\end{equation*}

\noindent is a \emph{polynomial} in $g$, whose leading coefficient is

\begin{equation*}
\frac{6^{|\vec{d}|}}{C(\vec{d})}
\end{equation*}

\noindent As far as the author is aware, this observation has not been fleshed out in the literature. We hope that this observation of polynomiality will contribute to the development of faster algorithms to compute $\left<\tau_{d_1}\ldots\tau_{d_n}\right>_g$. It may also contribute to a different understanding of the large-genus asymptotics of $\left<\tau_{d_1}\ldots\tau_{d_n}\right>_g$.

\subsection{Strategy of Proof}

\noindent The proof of Theorem \ref{MainTheorem} is in Section \ref{IPP}. There are three main components that comprise the argument. \\

\noindent Let $\vec{d} \in \mathbb{Z}^n_{\geq 0}$, and define

\begin{equation*}
L_{\vec{d}}(g) := 24^gg!C(\vec{d})\left<\tau_{\vec{d}}\tau_{3g - 2 + n - |\vec{d}|}\right>_g
\end{equation*} 

\noindent A priori, this is an arbitrary family of intersection numbers parametrized by genera $g$. However, we prove that $L_{\vec{d}}(g)$ is an \emph{integer-valued polynomial} in $g$ whose leading coefficient is $6^{|\vec{d}|}$. In order to prove this, we use the fact that $L_{\vec{d}}(g)$ can be recursively computed by the \emph{String Equation}, the \emph{Dilaton Equation}, and the (higher) \emph{Virasoro constraints}. We then prove that the property of being an integer-valued polynomial with leading coefficient $6^{|\vec{d}|}$ is a property that remains invariant under all three of these recursive operations. This concludes the first component of the argument. \\

\noindent Once we know that $L_{\vec{d}}(g)$ is an integer-valued polynomial, we can consider the shifted integer-valued polynomial $L_{\vec{d}}(g + m(\vec{d}))$ (see the paragraphs directly following the statement of Theorem \ref{MainTheorem} for an explanation of the shift $m(\vec{d})$). We then expand $L_{\vec{d}}(g + m(\vec{d}))$ in the binomial basis $\{{g - 1 \choose k}\}$. The choice of the binomial basis $\{{g - 1 \choose k}\}$ is not an arbitrary choice: in the field of Ehrhart theory, such an expansion of an Ehrhart polynomial computes the $f^*$-vector of the polynomial, which, under certain assumptions, gives one information about the geometry of the corresponding polytope (see Section \ref{Ehrhart} below). We prove that the $f^*$-vector of $L_{\vec{d}}(g + m(\vec{d}))$ is nonnegative. The strategy for this component of the argument is the same as in the previous one: we show that nonnegativitiy of the $f^*$-vector is a property that remains invariant throughout all recursive procedures that compute $L_{\vec{d}}(g + m(\vec{d}))$. \\

\noindent Once we know that the $f^*$-vector of $L_{\vec{d}}(g + m(\vec{d}))$ is nonnegative, we apply the following classification theorem of Breuer: \emph{a polynomial $P(g)$ of degree $d$ is the Ehrhart polynomial of a partial polytopal complex of dimension $d$ if and only if the $f^*$-vector of $P(g)$ is integral and nonnegative}. 

\subsection{Outline of Paper}

\noindent Our intention is to make the paper readable to combinatorialists from Ehrhart theory and algebraic geometers interested in $\mbar_{g, n}$. Consequently, we spend a good portion of the paper explaining fundamental ideas and results from both fields. However, we keep technical details to a minimum, and we refer the reader to sources in the literature when necessary. \\

In Section \ref{Ehrhart}, we discuss standard results and basic notions from Ehrhart theory. The main goal of Section \ref{Ehrhart} is to define what one means when one refers to the `Ehrhart polynomial of a partial polytopal complex'. If one is already familiar with what this means, it is safe to skip this section. \\

In Section \ref{Tautological}, we recall results concerning tautological intersection numbers, especially the ones needed for this paper. The goal of this section to show how one computes the tautological intersection number $\left<\tau_{d_1}\ldots\tau_{d_n}\right>_g$ recursively using the String/Dilaton equation and the Virasoro constraints. \\

In Section \ref{IPP} we present the proof of Theorem \ref{MainTheorem}. The main computational idea is to view the Virasoro constraints as a recursion for integer-valued polynomials in general, and Ehrhart polynomials in particular.  \\

The paper ends with the computation of a few examples (see (Section \ref{Examples}), along with an outline for future work (Section \ref{Future}).

\subsection{Ackowledgments}

The author is grateful to Renzo Cavalieri for his comments and suggestions on an earlier draft. Many thanks to David and Majka Phillips for their hospitality while the author was making edits.  \\

This work was supported by the Cluster of Excellence \emph{Mathematics Münster} and the CRC 1442 \emph{Geometry: Deformations and Rigidity}. ``Funded by the Deutsche Forschungsgemeinschaft (DFG, German Research Foundation) - Project-I 427320536 - SFB 1442, as well as under Germany's Excellence Strategy EXC 2044390685587, Mathematics Münster: Dynamics - Geometry - Structure"

\section{Ehrhart Theory}\label{Ehrhart}

The purpose of this section is to define the Ehrhart polynomial of a partial polytopal complex. Pedagogically, it seems natural to first begin with a discussion on polytopes, which will provide the necessary background to discuss partial polytopal complexes. \\

For more details concerning the Ehrhart theory of convex integral $d$-polytopes, we recommend (\cite{beck2007computing}, Chapter 3). For the most part, the standard techniques and ideas of Ehrhart theory in the context of polytopes extend to the context of partial polytopal complexes. For an exposition on partial polytopal complexes that aligns well with the purposes of  this paper, see  (\cite{breuer2012ehrhart}, Section 2). \\

\subsection{Ehrhart Polynomials of Convex Integral $d$-Polytopes}

\noindent Let $P \subset \mathbb{R}^d$ be an integral convex polytope, and let $v_1, \ldots, v_n \in \mathbb{Z}^d$ be the vertices of $P$, 

\begin{equation*}
P = \text{Conv}(v_1, \ldots, v_n) \subset \mathbb{R}^d
\end{equation*}

\noindent Throughout, we always assume that $P$ is \emph{full-dimensional}, that is, $P$ is a $d$-polytope. For an integer $g \geq 1$, the $g^{th}$-dilate of $P$, denoted $gP$, is defined to be 

\begin{equation*}
gP := \text{Conv}(gv_1, \ldots, gv_n) = \{gp : p \in P\}
\end{equation*}

\noindent Ehrhart theory is chiefly concerned with the task of counting \emph{lattice points} in $gP$, i.e. understanding the function

\begin{equation*}
L_P(g) : g \mapsto |gP \cap \mathbb{Z}^d|
\end{equation*}

\noindent Here is the main theorem of Ehrhart theory:

\begin{Theorem}[Ehrhart's Theorem]\label{EhrhartTheorem}
$L_P(g)$ is a rational polynomial in $g$ of degree $d$.
\end{Theorem}

\noindent We call $L_P(g)$ the \emph{Ehrhart polynomial} of $P$. Theorem \ref{EhrhartTheorem} is originally due to Eug\`{e}ne Ehrhart. For a proof, see (\cite{beck2007computing}, Theorem 3.8). \\

\noindent An important family of polytopes are the $d$-dimensional \emph{simplices}. An integral \emph{$d$-simplex} is a $d$-dimensional polytope $\Delta \subset \mathbb{R}^d$ that is the convex hull of $d + 1$ affinely independent integral points. The \emph{standard $d$-simplex} is given by $\Delta_d := \text{Conv}(\vec{0}, e_1, \ldots, e_d) \subseteq \mathbb{R}^d$. An \emph{open $d$-simplex} is the relative interior of a $d$-simplex.

\begin{Example}\label{runningpolytopeexample}

\noindent Let $P \subset \mathbb{R}^2$ be the standard $2$-simplex,

\begin{equation*}
P = \text{Conv}(\vec{0}, e_1, e_2) \subset \mathbb{R}^2
\end{equation*}

\noindent By Ehrhart's theorem, we know there exists rational numbers $a_0, a_1$, and $a_2$ such that

\begin{equation*}
L_P(g) = a_2g^2 + a_1g + a_0
\end{equation*}

\noindent Counting lattice points by hand (see Figure \ref{runningfigure}), we see that $L_P(1) = 3, L_P(2) = 6$, and $L_P(3) = 10$. These equations suffice to determine the coefficients $a_0, a_1$, and $a_2$, and we obtain

\begin{equation*}
L_P(g) = \frac{1}{2}g^2 + \frac{3}{2}g + 1
\end{equation*}

\end{Example}

\begin{figure}
\begin{center}
\begin{tikzpicture}
\draw (0, 0) -- (1, 0) -- (0, 1) -- (0, 0);
\draw (2, 0) -- (4, 0) -- (2, 2) -- (2, 0);
\draw (5, 0) -- (8, 0) -- (5, 3) -- (5, 0);

\filldraw (0, 0) circle (2pt);
\filldraw (1, 0) circle (2pt);
\filldraw (0, 1) circle (2pt);

\draw (0.5, -0.5) node {$L_P(1) = 3$};

\filldraw (2, 0) circle (2pt);
\filldraw (3, 0) circle (2pt);
\filldraw (4, 0) circle (2pt);
\filldraw (2, 1) circle (2pt);
\filldraw (3, 1) circle (2pt);
\filldraw (2, 2) circle (2pt);

\draw (3, -0.5) node {$L_P(2) = 6$};

\filldraw (5, 0) circle (2pt);
\filldraw (6, 0) circle (2pt);
\filldraw (7, 0) circle (2pt);
\filldraw (8, 0) circle (2pt);
\filldraw (5, 1) circle (2pt);
\filldraw (6, 1) circle (2pt);
\filldraw (7, 1) circle (2pt);
\filldraw (5, 2) circle (2pt);
\filldraw (6, 2) circle (2pt);
\filldraw (5, 3) circle (2pt);

\draw (6.5, -0.5) node {$L_P(3) = 10$};

\end{tikzpicture}
\end{center}
\caption{As is in Example \ref{runningpolytopeexample}, we can compute $L_P(g)$ by polynomial interpolation}
\label{runningfigure}

\end{figure}
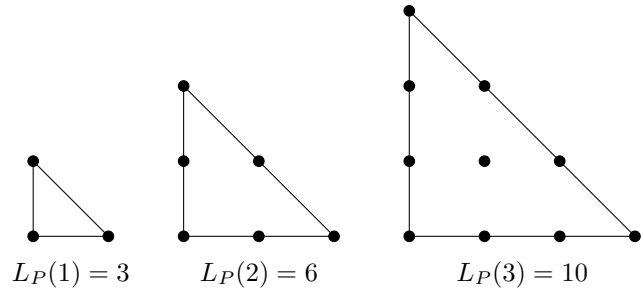

We need a slight generalization of integral $d$-polytopes. Instead of dealing with one polytope at a time, it is possible to take a collection of polytopes and glue them along their faces, albeit in a compatible way. We call these objects \emph{polytopal complexes} (see also \cite{ziegler2012lectures}, Definition 5.1):

\begin{Definition}

\noindent A \emph{polytopal complex} is a finite collection $K$ of polytopes that satisfies the following three properties:

\begin{enumerate}
\item{The empty polytope is in $K$}
\item{$P \in K, f\subseteq P$ is a face of P $\implies f \in K$ }
\item{$P, Q \in K \implies P\cap Q \in K$, and $P\cap Q$ is a face of both $P$ and $Q$}
\end{enumerate}

\noindent The elements of $K$ are called the \emph{faces} of $K$. The dimension of $K$ is the maximum dimension of the faces of $K$.

\end{Definition}

\noindent Ehrhart's theorem still holds in the context of polytopal complexes, that is, the counting function $L_K(g) = \#\{\text{integer lattice points in $gK$}\}$ is a rational polynomial in $g$ of degree $d$. \\

\noindent Let $P, Q \subseteq \mathbb{R}^d$ be integral $d$-polytopes. We say $P$ and $Q$ are \emph{lattice equivalent} if there exists an affine isomorphism $\phi: \mathbb{Z}^d \rightarrow \mathbb{Z}^d$ sending the vertices of $P$ to the vertices of $Q$. A polytopal complex $K$ is a \emph{simplicial complex} if every face of $K$ is a simplex. A \emph{triangulation} of an integral polytope $P$ is a simplicial complex whose support is $P$. We say that a triangulation is \emph{unimodular} if the simplices in the corresponding simplicial complex are lattice equivalent to the standard simplex.  \\

\noindent Unimodular triangulations are useful in Ehrhart theory due to the following result:

\begin{Theorem}
Let $(P, K)$ be a unimodular triangulation of an integral $d$-polytope $P$. Define
\begin{equation*}
f_i^* := \#\{\text{$i$-dimensional open simplices in $K$}\}
\end{equation*}
Then the Ehrhart polynomial of $P$ has the following form:
\begin{equation*}
L_P(g) := \sum_{k = 0}^df_k^*{g - 1 \choose k} 
\end{equation*}
\end{Theorem}

\noindent We define the \emph{$f^*$-vector} of $P$ to be the vector of integers $(f_0^*, \ldots, f_d^*)$.

\begin{Example}
Recall the polytope $P$ given in Example \ref{runningexample}, $P = \text{Conv}(\vec{0}, e_1, e_2)$. This is just the standard two-dimensional simplex. The simplex itself provides a unimodular triangulation. Upon inspection we see that its $f^*$-vector is $(3, 3, 1)$. Therefore,
\begin{align*}
L_P(g) & = 3{g - 1 \choose 0} + 3{g - 1 \choose 1} + {g - 1 \choose 2} \\
& = 3 + 3(g - 1) + \frac{1}{2}(g - 1)(g - 2) \\
& = 3g + \frac{1}{2}(g^2 - 3g + 2) \\
& = \frac{1}{2}g^2 + \frac{3}{2}g + 1
\end{align*} 
\noindent as expected.
\end{Example}

\noindent Not every integral polytope $P$ admits a unimodular triangulation. However, it always makes sense to talk about the $f^*$-vector of an integral polytope. The reason is as follows. Suppose $L(g) \in \mathbb{Q}[g]$ is a polynomial of degree $d$. The set $\{{g - 1 \choose k}\}_{k = 0}^d$ forms $\mathbb{Q}$-basis for the vector space of all rational polynomials of degree $d$. Therefore, there exists a unique vector $(f_0^*, \ldots, f_d^*) \in \mathbb{Q}^{d + 1}$ such that $L(g) = \sum_{k = 0}^df_k^*{g - 1 \choose k}$. The $f^*$-vector of an integral polytope $P$ is the unique vector $(f_0^*,\ldots f_d^*) \in \mathbb{Q}^{d + 1}$ such that $L_P(g) = \sum_{k = 0}^df_k^*{g - 1 \choose k}$. \\

\noindent However, notice that $L_P(g)$ is actually an \emph{integer-valued polynomial}:

\begin{Definition}
An \emph{integer-valued polynomial} $L(g) \in \mathbb{Q}[g]$ is a polynomial such that $L(\mathbb{N}) \subseteq \mathbb{Z}$.
\end{Definition}

\noindent It turns out that the $f^*$-vector of an integer-valued polynomial is always integral. Therefore, we can make the following definition:

\begin{Definition}
Let $P$ be a convex integral $d$-polytope. The $f^*$-vector of $P$ is the unique integer tuple $(f_0^*, \ldots, f_d^*) \in \mathbb{Z}^{d + 1}$ such that
\begin{equation*}
L_P(g) = \sum_{k = 0}^df_k^*{g - 1 \choose k}
\end{equation*} 
\end{Definition}

\subsection{Ehrhart Polynomials of Partial Polytopal Complexes}

\noindent Generalizing even further, we need to consider $\emph{open $d$-polytopes}$. An \emph{open $d$ polytope} is the relative interior of an integral $d$-polytope. This brings us to the generalization of polytopal complexes that we need:

\begin{Definition}
An \emph{integral partial polytopal complex} $K$ is the disjoint union of a finite collection of open integral polytopes. The elements of $K$ are called the faces of $K$. The dimension $d$ of $K$ is the maximum dimension of the faces of $K$. The Ehrhart polynomial of $K$, denoted $L_K(g)$, is the sum of the Ehrhart polynomials of each face of $K$.
\end{Definition}

\begin{Remark}
Notice that an integral partial polytopal complex is a generalization of an integral polytopal complex. This follows from the observation that the support of an integral polytopal complex is the disjoint union of the relative interiors of all of its faces. Thus, every polytopal complex is neccessarily a partial polytopal complex. Intuitively, one can think of this generalization as simply allowing oneself to `excise' or `throw away' faces of any polytope $P \in K$ in a polytopal complex.
\end{Remark}

\begin{Definition}
Let $K$ be a partial polytopal complex of dimension $d$. A \emph{triangulation} $\mathcal{T}$ of $K$ is a disjoint union of open simplices whose support is $K$. The triangulation $\mathcal{T}$ is \emph{unimodular} if the closure of each open simplex in $\mathcal{T}$ is lattice equivalent to the standard simplex. The $f^*$-vector of $\mathcal{T}$ is $(f_0^*, \ldots f_d^*)$, where $f_i^* := \#\{\text{$i$-dimensional open simplices in $\mathcal{T}$}\}$. 
\end{Definition}

\noindent As in the case of polytopal complexes, if one can find a unimodular triangulation of a partial polytopal complex, this immediately gives its Ehrhart polynomial:

\begin{Theorem}\label{unimodular}
Let $K$ be a partial polytopal complex of dimension $d$ and let $\mathcal{T}$ be a unimodular triangulation of $K$. Then
\begin{equation*}
L_K(g) = \sum_{i = 0}^d f_i^*{g - 1 \choose i}
\end{equation*}
\noindent where $(f_0^*, \ldots, f_d^*)$ is the $f^*$-vector of $\mathcal{T}$.
\end{Theorem}

\noindent Even if a partial polytopal complex does not admit a unimodular triangulation, it still makes sense to talk about its $f^*$-vector. Furthermore, Ehrhart polynomials of partial polytopal complexes are completely classified by their $f^*$-vector due to the following result of Breuer:

\begin{Theorem}[\cite{breuer2012ehrhart}, Theorem 2]\label{Breuer}
Let $P(g)$ be an integer-valued polynomial of degree $d$. Then $P(g)$ is the Ehrhart polynomial of an integral partial polytopal complex if and only if the $f^*$-vector $(f_0^*, \ldots, f_d^*)$ of $P(g)$ is non-negative i.e. $f_i^* \geq 0$ for all $0 \leq i \leq d$. 
\end{Theorem}

\subsection{Useful Properties of Ehrhart Polynomials}

We recall properties of Ehrhart polynomials of partial polytopal complexes that will be useful in the proof of Theorem \ref{MainTheorem}.

\begin{Lemma}\label{shift}
Let $K$ be a partial polytopal complex of dimension $d$, and let $L_K(g)$ be its Ehrhart polynomial. Then, for any $k \geq 0$, there exists a partial polytopal complex $K'$ such that $L_K(g + k) = L_{K'}(g)$ is the Ehrhart polynomial of $K'$.
\end{Lemma}

\begin{proof}
This is a direct application of Theorem \ref{Breuer} and the basic binomial identity ${g  \choose k} = {g - 1 \choose k} + {g - 1 \choose k - 1}$. Indeed, by Theorem \ref{Breuer}, there exists a non-negative $f^*$-vector $(f_0^*, \ldots, f_d^*)$ such that
\begin{equation*}
L_K(g) = \sum_{k = 0}^df_k^*{g - 1 \choose k}
\end{equation*}
\noindent and therefore,
\begin{align*}
L_K(g + 1) & = \sum_{k = 0}^df_k^*{g \choose k} \\
& = f_0^*{g - 1 \choose 0} + \sum_{k = 1}^df_k^*{g \choose k} \\
& = f_0^*{g - 1 \choose 0} + \sum_{k = 1}^d f_k^*\left({g - 1 \choose k} + {g - 1 \choose k - 1}\right)  
\end{align*}
\noindent By theorem \ref{Breuer}, $L_K(g + 1)$ is the Ehrhart polynomial of some partial polytopal complex $K'$. Therefore, the desired result is true for $k = 1$, and we proceed by induction.
\end{proof}

\noindent The next property concerns the cartesian products of partial polytopal complexes. To get a better understanding of how this works, it makes sense to first start with taking products of polytopes, and then generalize. 

\begin{Definition}
Let $P \subset \mathbb{R}^{d_1}$ and $Q \subset \mathbb{R}^{d_2}$ be integral polytopes of dimension $d_1$ and $d_2$, respectively. The \emph{cartesian product}, or simply \emph{product} of the polytopes $P$ and $Q$ is the $(d_1 + d_2)$-dimensional integral polytope denoted by
\begin{align*}
P \times Q & := \{(p, q) \in \mathbb{R}^{d_1} \times \mathbb{R}^{d_2} = \mathbb{R}^{d_1 + d_2} : p \in P, q \in Q\} \\
& = \text{Conv}\{(p, q) \in \mathbb{R}^{d_1 + d_1} : p \ \text{a vertex of $P$}, \ q \ \text{a vertex of $Q$} \}
\end{align*}
The product of two open polytopes is defined similarly.
\end{Definition}

\noindent The operation of taking products of (open) polytopes plays well with Ehrhart polynomials in that, if $P$ and $Q$ are integral polytopes,

\begin{equation*}
L_P(g) \times L_Q(g) = L_{P \times Q}(g)
\end{equation*}

\noindent Furthermore, we can also define the product of two partial polytopal complexes $K_1$ and $K_2$ in a way that also plays well with taking their Ehrhart polynomials.

\begin{Definition}
Let $K_1$ and $K_2$ be partial polytopal complexes. The product $K_1 \times K_2$ is the partial polytopal complex defined by

\begin{align*}
K_1 \times K_2 & := \{(p, q) \in P \times Q : P \in K_1, Q \in K_2\} \\
& = \coprod_{P \in K_1, Q \in K_2}(P \times Q)
\end{align*}

\end{Definition}

\begin{Lemma}\label{product}
Let $K_1$ and $K_2$ be partial polytopal complexes, and let $L_{K_1}(g)$ and $L_{K_2}(g)$ be their Ehrhart polynomials, respectively. Then
\begin{equation*}
L_{K_1}(g) \times L_{K_2}(g) = L_{K_1 \times K_2}(g)
\end{equation*}
\end{Lemma}

\begin{proof}
The Ehrhart polynomial of $K_1 \times K_2$ is, by definition, the sum of the Ehrhart polynomials of the faces of $K_1 \times K_2$. But the faces of $K_1 \times K_2$ are the open polytopes $\{P \times Q : P \in K_1, Q \in K_2\}$. Therefore
\begin{align*}
L_{K_1 \times K_2}(g) & = \sum_{P \in K_1, Q \in K_2}L_{P \times Q}(g) \\
& = \sum_{P \in K_1, Q \in K_2}L_P(g) \times L_Q(g) \\
& = \left(\sum_{P \in K_1}L_P(g)\right)\left(\sum_{Q \in K_2}L_Q(g)\right) \\
& = L_{K_1}(g) \times L_{K_2}(g)
\end{align*}  
\end{proof}

\begin{Lemma}\label{volume}
Let $K$ be a partial polytopal complex, and let $L_K(g)$ be its Ehrhart polynomial. Then the leading coefficient of $L_K(g)$ is the (Euclidean) volume of $K$.
\end{Lemma}

\begin{proof}
In the case that $K$ is an integral $d$-polytope, the statement is classical (see \cite{beck2007computing}, Corollary 3.20). The result then follows from the observation that the volume of an integral $d$-polytope is the same as the volume of its relative interior.
\end{proof}

\section{Tautological Intersection Numbers on $\mbar_{g, n}$}\label{Tautological}

In this section, we introduce $\mbar_{g, n}$, its tautological ring $R^*(\mbar_{g, n})$, and all of the necessary results that are required to compute intersection numbers of the form

\begin{equation*}
\left<\tau_{d_1}\ldots\tau_{d_n}\right>_g := \int_{\mbar_{g, n}}\psi_1^{d_1}\ldots\psi_n^{d_n}
\end{equation*}

\noindent For readers who would like a nice introduction to the tautological intersection theory of $\mbar_{g, n}$, especially one that focuses on computation, we would recommend \cite{zvonkine2014introduction} and \cite{pandharipande2015calculus}.

\subsection{The Moduli of Stable Curves}

Let $(g, n)$ be a pair of nonnegative integers such that $2g - 2 + n > 0$, and let $(C, p_1, \ldots, p_n)$ be an at-worst nodal curve of genus $g$, along with $n$ smooth marked points $p_1, \ldots, p_n$. We say that $(C, p_1, \ldots, p_n)$ is \emph{stable} if the automorphism group of $(C, p_1, \ldots, p_n)$ is finite. Alternatively, we say $(C, p_1, \ldots, p_n)$ is stable if it satisfies the following two conditions:

\begin{enumerate}
\item{If $C_0 \subseteq C$ is a rational irreducible component, then $C_0$ is incident to at least 3 `special points', that is, nodes or marked points}
\item{If $C_1 \subseteq C$ is an elliptic irreducible component, then $C_1$ is incident to at least 1 `special point'}
\end{enumerate}

\noindent We denote by $\mbar_{g, n}$ the moduli space of stable $n$-pointed genus $g$ curves. It is a smooth Deligne-Mumford stack of dimension $\text{dim}(\mbar_{g, n}) = 3g - 3 + n$. The universal curve is denoted

\begin{equation*}
\pi_{n + 1} : \mbar_{g, n + 1} \rightarrow \mbar_{g, n}
\end{equation*}

\noindent We often call $\pi_{n + 1}$ the \emph{forgetful morphism}: it sends a point $[C, p_1, \ldots, p_{n + 1}]$ to $[C, p_1, \ldots, p_n]$, that is, it `forgets' the last marked point. However, in order for this map to be well defined, one must contract any rational components that are unstable. Similarly, we define $\pi_i : \mbar_{g, n + 1} \rightarrow \mbar_{g, n}$ to be the morphism that forgets the $i^{th}$ marked point.  \\

\noindent Sitting inside $\mbar_{g, n}$ is a dense open locus of curves 

\begin{equation*}
\mathcal{M}_{g, n} \subset \mbar_{g, n}
\end{equation*}

\noindent parametrizing smooth $n$-pointed genus $g$ curves. The \emph{boundary} of $\mbar_{g, n}$, which is the complement of $\mathcal{M}_{g, n}$ in $\mbar_{g, n}$, parametrizes nodal stable curves. The boundary of $\mbar_{g, n}$ has a \emph{stratification}, and each stratum parametrizes curves of a fixed topological type. 

\begin{Example}\label{runningexample}
Consider the moduli space $\mbar_{2, 4}$. As a running example, we consider two strata $\Gamma_2 \subset \Gamma_1 \subset \mbar_{2, 4}$. \\

\noindent Let $\Gamma_1 \subset \mbar_{2, 4}$ be the stratum parametrizing curves of the following topological type: a generic curve $(C, p_1, p_2, p_3, p_4) \in \Gamma_1$ is a curve $C = C_1\cup C_1$, where $C_1$ and $C_2$ are stable curves of genus 1 attached via a node, $\{p_1, p_2\} \subset C_1$, and $\{p_3, p_4\} \subset C_2$. Below is a drawing of a generic curve in $\Gamma_1$ as a two-dimensional nodal topological surface:

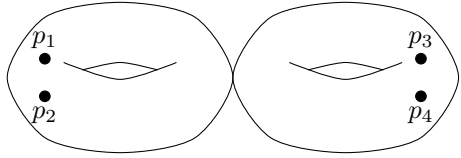
\begin{center}
\begin{tikzpicture}

\draw plot [smooth cycle] coordinates {(0, 0) (0.5, -0.8) (1.5, -1) (2.5, -0.8) (3, 0) (2.5, 0.8) (1.5, 1) (0.5, 0.8)};
\draw plot [smooth cycle] coordinates {(3, 0) (3.5, -0.8) (4.5, -1) (5.5, -0.8) (6, 0) (5.5, 0.8) (4.5, 1) (3.5, 0.8)};

\draw plot [smooth] coordinates {(0.75, 0.2) (1.5, -0.025) (2.25, 0.2)};
\draw plot [smooth] coordinates {(1, 0.1) (1.5, 0.2) (2, 0.1)};

\draw plot [smooth] coordinates {(3.75, 0.2) (4.5, -0.025) (5.25, 0.2)};
\draw plot [smooth] coordinates {(4, 0.1) (4.5, 0.2) (5, 0.1)};

\filldraw (0.5, 0.25) circle (2pt);
\draw (0.5, 0.5) node {$p_1$};
\filldraw (0.5, -0.25) circle (2pt);
\draw (0.5, -0.5) node {$p_2$};

\filldraw (5.5, 0.25) circle (2pt);
\draw (5.5, 0.5) node {$p_3$};
\filldraw (5.5, -0.25) circle (2pt);
\draw (5.5, -0.5) node {$p_4$};

\end{tikzpicture}
\end{center}


\noindent Similarly, let $\Gamma_2 \subset \Gamma_1 \subset \mbar_{2, 4}$ be the stratum parametrizing curves of the following topological type: a generic curve $(C, p_1, p_2, p_3, p_4) \in \Gamma_2$ is a curve $C = C_1 \cup \tilde{C_2}$, where $C_1$ is a stable curve of genus 1, $\tilde{C_2}$ is a nodal curve of genus 1, $\{p_1, p_2\} \subset C_1$, and $\{p_3, p_4\} \subset \tilde{C_2}$. Below is a drawing of a generic curve in $\Gamma_2$ as a two-dimensional nodal topological surface:

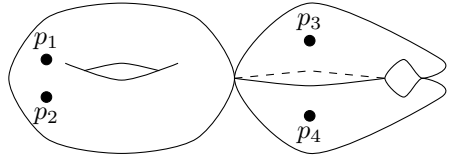
\begin{center}
\begin{tikzpicture}

\draw plot [smooth cycle] coordinates {(0, 0) (0.5, -0.8) (1.5, -1) (2.5, -0.8) (3, 0) (2.5, 0.8) (1.5, 1) (0.5, 0.8)};

\filldraw (0.5, 0.25) circle (2pt);
\draw (0.5, 0.5) node {$p_1$};
\filldraw (0.5, -0.25) circle (2pt);
\draw (0.5, -0.5) node {$p_2$};

\filldraw (4, 0.5) circle (2pt);
\draw (4, 0.75) node {$p_3$};
\filldraw (4, -0.5) circle (2pt);
\draw (4, -0.75) node {$p_4$};

\draw plot [smooth] coordinates {(0.75, 0.2) (1.5, -0.025) (2.25, 0.2)};
\draw plot [smooth] coordinates {(1, 0.1) (1.5, 0.2) (2, 0.1)};

\draw plot [smooth cycle] coordinates {(3, 0) (4, -1) (5.75, -0.25) (5.5, 0) (5.25, -0.25) (5, 0) (5.25, 0.25) (5.5, 0) (5.75, 0.25) (4, 1)};

\draw plot [smooth] coordinates {(3, 0) (4, -0.1) (5, 0)};
\draw[dashed] (3, 0) -- (4, 0.1) -- (5, 0);

\end{tikzpicture}
\end{center}

\end{Example}

\noindent The boundary strata in $\mbar_{g, n}$ can be indexed by \emph{stable graphs}.

\begin{Definition}[\cite{pandharipande2015calculus}, Section 5.2]
A stable graph $\Gamma$ is the data

\begin{equation*}
\Gamma = (V, H, L, g : V \rightarrow \mathbb{Z}_{\geq 0}, v : H \rightarrow V, \iota : H \rightarrow H)
\end{equation*}

\noindent where,

\begin{enumerate}
\item{$V$ is a vertex set, and $g : V \rightarrow \mathbb{Z}_{\geq 0}$ is the genus assignment}
\item{$H$ is a set of half-edges, $v: H \rightarrow V$ is a vertex assignment i.e. indicates which vertex each half-edge is incident to, and $\iota : H \rightarrow H$ is an involution that indicates when two half edges are glued together}
\item{$E$ is an edge set, determined by the $2$-cycles of $\iota$}
\item{$L$ is a set of legs, determined by the fixed points of $\iota$; it is in bijection with the set of markings $\{1, \ldots, n\}$} 
\item{The data $(V, E)$ defines a connected graph.}
\item{For each vertex $v \in V$, $2g(v) - 2 + n(v) > 0$, where $n(v)$ is the valence of $\Gamma$ at v including both edges and legs}
\end{enumerate}

\end{Definition} 

\begin{Example}\label{runningexamplegraphs}
The stratum $\Gamma_1 \subset \mbar_{2, 4}$ from Example \ref{runningexample} corresponds to the following stable graph:

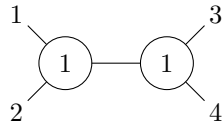
\begin{center}
\begin{tikzpicture}

\draw (0, 0) circle (10pt);
\draw (0.35, 0) -- (1, 0);
\draw (1.35, 0) circle (10pt);

\draw (-0.25, 0.25) -- (-0.5, 0.5);
\draw (-0.65, 0.65) node {$1$};
\draw (-0.25, -0.25) -- (-0.5, -0.5);
\draw (-0.65, -0.65) node {$2$};

\draw (1.6, 0.25) -- (1.85, 0.5);
\draw (2, 0.65) node {$3$};
\draw (1.6, -0.25) -- (1.85, -0.5);
\draw (2, -0.65) node {$4$};

\draw (0, 0) node {$1$};
\draw (1.35, 0) node {$1$};

\end{tikzpicture}
\end{center}

\noindent The above stable graph consists of two vertices, both with genus assignment $1$. Each vertex is incident to three half edges. The only $2$-cycle of the involution $\iota$ corresponds to the two half edges that glue together to form the edge connecting both vertices. The set of legs $L$ corresponds to the half-edges labelled $1, 2, 3$, and $4$. \\

\noindent The stratum $\Gamma_2 \subset \mbar_{2, 4}$ from Example \ref{runningexample} corresponds to the following stable graph:

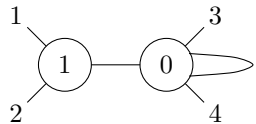
\begin{center}
\begin{tikzpicture}

\draw (0, 0) circle (10pt);
\draw (0.35, 0) -- (1, 0);
\draw (1.35, 0) circle (10pt);

\draw (-0.25, 0.25) -- (-0.5, 0.5);
\draw (-0.65, 0.65) node {$1$};
\draw (-0.25, -0.25) -- (-0.5, -0.5);
\draw (-0.65, -0.65) node {$2$};

\draw (1.6, 0.25) -- (1.85, 0.5);
\draw (2, 0.65) node {$3$};
\draw (1.6, -0.25) -- (1.85, -0.5);
\draw (2, -0.65) node {$4$};

\draw (0, 0) node {$1$};
\draw (1.35, 0) node {$0$};

\draw plot [smooth] coordinates {(1.65, 0.15) (2.25, 0.1) (2.5, 0) (2.25, -0.1) (1.65, -0.15)};

\end{tikzpicture}
\end{center}

\noindent Notice that for this stable graph, the vertex with genus assignment $0$ has a self-edge, corresponding to the self-node.

\end{Example}

\noindent Due to the dictionary between stable graphs and boundary strata, we unambiguously refer to a boundary stratum by its stable graph. In the case of our running example, $\Gamma_1$ will mean `the stable graph corresponding to the boundary stratum $\Gamma_1$'. Similarly, $\Gamma_2$ will mean `the stable graph corresponding to the boundary stratum $\Gamma_2$. \\

\noindent Every stable graph $\Gamma$ corresponds to a product of moduli spaces:

\begin{equation*}
\mbar_{\Gamma} := \prod_{v \in V}\mbar_{g(v), n(v)}
\end{equation*}

\noindent For instance, in our running example, we have

\begin{align*}
\mbar_{\Gamma_1} & = \mbar_{1, 3} \times \mbar_{1, 3} \\
\mbar_{\Gamma_2} & = \mbar_{1, 3} \times \mbar_{0, 5}
\end{align*}

\noindent For every stable graph $\Gamma$, there exists a canonical morphism

\begin{equation*}
\xi_\Gamma : \mbar_\Gamma \rightarrow \mbar_{g, n}
\end{equation*}

\noindent whose image is the boundary stratum corresponding to $\Gamma$ (see \cite{pandharipande2015calculus}, Section 5.2 for details).

\subsection{The Tautological Ring} 

\begin{Definition}\label{tautological}
The \emph{tautological ring} $R^*(\mbar_{g, n}) \subset A^*(\mbar_{g, n})$ is the smallest $\mathbb{Q}$-subalgebra of $A^*(\mbar_{g, n})$ closed under pushforwards of the morphisms $\pi_i$ and $\xi_\Gamma$. Elements in $R^*(\mbar_{g, n})$ are called tautological classes.
\end{Definition}

\noindent It turns out that many important cycles in $A^*(\mbar_{g, n})$ are tautological, and this has prompted an intensive investigation of this ring in recent times. However, it is difficult to gain access to the tautological ring using only its definition. Fortunately, there is a nice result due to Graber and Pandharipande that gives an explicit additive generating set for $R^*(\mbar_{g, n})$ (see Theorem \ref{additive} below). In order to state this result, we need to define two types of Chow classes in $A^*(\mbar_{g, n})$, the $\psi$-classes, and the $\kappa$-classes.

\begin{Definition}
For $1 \leq i \leq n$, let $\mathbb{L}_i$ the line bundle on $\mbar_{g, n}$ whose fiber over a point $[C, p_1, \ldots, p_n] \in \mbar_{g, n}$ is $T_{p_i}^*C$, the cotangent space to $C$ at the $i^{th}$ marked point. For $1 \leq i \leq n$ and $0 \leq m \leq 3g - 3 + n$, define

\begin{align*}
\psi_i & := c_1(\mathbb{L}_i) \in A^1(\mbar_{g, n}) \\
\kappa_m & := {\pi_{n + 1}}_*\left(\psi_{n + 1}^{m + 1}\right) \in A^m(\mbar_{g, n})
\end{align*}

\end{Definition}

\begin{Theorem}[\cite{graber2003constructions}, Proposition 11]\label{additive}
The set
\begin{equation*}
\left\{{\xi_{\Gamma}}_*\left(\prod_{v \in V}\theta_v\right)\right\}
\end{equation*}
\noindent where $\Gamma$ is a stable graph, and $\theta_v$ is a monomial of $\psi$-classes and $\kappa$-classes on $\mbar_{g(v), n(v)}$, forms an additive generating set of $R^*(\mbar_{g, n})$.
\end{Theorem}

\noindent A direct consequence of Theorem \ref{additive} is that, any tautological intersection number can be reduced to intersection numbers involving only $\psi$-classes and $\kappa$-classes. However, one can do even better than this: tautological intersection numbers involving $\kappa$-classes can be reduced to intersection numbers only involving $\psi$-classes. This is due to the following result:

\begin{Proposition}[\cite{zvonkine2014introduction}, Corollary 3.23]\label{eliminatekappa}
Let $Q$ be a polynomial in the variables $\kappa_m, \psi_1, \ldots, \psi_n$, and let $\tilde{Q}$ be the polynomial obtained from Q by the substitution $\kappa_i \mapsto \kappa_i - \psi_{n + 1}^i$. Then
\begin{equation*}
\int_{\mbar_{g, n}}\kappa_mQ = \int_{\mbar_{g, n + 1}}\psi_{n + 1}^{m + 1}\tilde{Q}
\end{equation*}
\end{Proposition} 

\subsection{Computing $\left<\tau_{d_1}\ldots\tau_{d_n}\right>_g$}

\noindent Since all tautological intersection numbers reduce to intersection numbers only involving $\psi$-classes, we can restrict our attention to the rational numbers

\begin{equation}
\left<\tau_{d_1}\ldots\tau_{d_n}\right>_g := \int_{\mbar_{g, n}}\psi_1^{d_1}\ldots\psi_n^{d_n} 
\end{equation}

\noindent It now remains to compute these numbers. When $g = 0$, there is a closed form expression in terms of multinomial coefficients:

\begin{equation}\label{genuszero}
\left<\tau_{d_1}\ldots\tau_{d_n}\right>_0 = {n - 3 \choose d_1\ldots d_n}
\end{equation}

\noindent When a $\psi$-class at only one marked point occurs, there is also a closed form expression:

\begin{Lemma}\label{onepointed}

For $n \geq 1$, 

\begin{equation*}
\left<\tau_{3g - 3 + n}\tau_0^{n - 1}\right>_g = \frac{1}{24^gg!}
\end{equation*}

\end{Lemma}

\begin{proof}
The proof can be found in (\cite{kock2001notes} Section 3.3, Example 3.3.5)
\end{proof}

\noindent For $g > 0$, there are various recursions that completely determine $\left<\tau_{d_1}\ldots\tau_{d_n}\right>_g$. First, we recall the \emph{String Equation} and the \emph{Dilaton Equation}:

\begin{equation*}
\text{\emph{The String Equation:}} \ \left<\tau_{d_1}, \ldots, \tau_{d_n}\tau_0\right>_g = \sum_{i = 1}^n\left<\tau_{d_1}\ldots\tau_{d_i - 1}\ldots\tau_{d_n}\right>_g
\end{equation*}

\begin{equation*}
\text{\emph{The Dilaton Equation:}} \ \left<\tau_{d_1}\ldots\tau_{d_n}\tau_1\right>_g = (2g - 2 + n)\left<\tau_{d_1}\ldots\tau_{d_n}\right>_g
\end{equation*}

\noindent When one wants to compute $\left<\tau_{d_1}\ldots\tau_{d_n}\right>_g$, and $d_i \leq 1$ for all $i$, the String Equation and Dilaton Equation will suffice. However, if $d_i \geq 2$ for some $i$, one needs the so-called \emph{Virasoro constraints}:

\begin{Theorem}[Virasoro Constraints]\label{virasoro}
Let $m \geq 1$. Then

\begin{align*}
(2m + 3)!!\left<\tau_{m + 1}\tau_{d_1}\ldots\tau_{d_n}\right>_g = & \sum_{i = 1}^n\frac{(2d_i + 1 + 2m)!!}{(2d_i - 1)!!}\left<\tau_{d_1}\ldots\tau_{d_i + m}\ldots\tau_{d_n}\right> \\
& + \frac{1}{2}\sum_{a + b = m - 1}(2a + 1)!!(2b + 1)!!\left<\tau_a\tau_b\tau_{d_1}\ldots\tau_{d_n}\right>_{g - 1} \\
& + \frac{1}{2}\sum_{\substack{a + b = m - 1 \\ I \amalg J = \{1, \ldots, n\} \\ g_1 + g_2 = g}}(2a + 1)!!(2b + 1)!!\left<\tau_a\prod_{i \in I}\tau_{d_i}\right>_{g_1}\left<\tau_b\prod_{i \in J}\tau_{d_i}\right>_{g_2}
\end{align*}

\end{Theorem}

\begin{proof}
See \cite{zvonkine2014introduction}, Section 4.2
\end{proof}

\noindent So, in summary, with the closed from expression in genus zero (Equation \ref{genuszero}), the String Equation, the Dilaton Equation, the closed formula for intersection numbers with $\psi$-classes supported at one marked point (Lemma \ref{onepointed}), and the Virasoro constraints (Theorem \ref{virasoro}), one can compute any intersection number $\left<\tau_{d_1}\ldots\tau_{d_n}\right>_g$.

\section{Proof of Theorem \ref{MainTheorem}}\label{IPP}

\noindent Throughout this section, we use the following shorthand notation:

\begin{itemize}
\item{$\displaystyle \vec{d} := (d_1, \ldots, d_n) \in \mathbb{Z}^n_{\geq 0}$}
\item{$\displaystyle |\vec{d}| := \sum_{i = 1}^n d_i$}
\item{$\displaystyle C(\vec{d}) := \prod_{i = 1}^n(2d_i + 1)!!$}
\item{$\tau_{\vec{d}} := \tau_{d_1}\ldots\tau_{d_n}$}
\item{$d_{n + 1} = d_{n + 1}(g, |\vec{d}|) := 3g - 2 + n - |\vec{d}|$}
\item{$\displaystyle L_{\vec{d}}(g) := 24^gg!C(\vec{d})\left<\tau_{\vec{d}}\tau_{d_{n + 1}}\right>_g$}
\item{$m = m(\vec{d}) := \left\lceil \frac{2 - n + |\vec{d}|}{3} \right\rceil - 1$}
\end{itemize}

\noindent In the case that $n = 0$, $\vec{d}$ is the empty vector, which we denote by $\vec{d} = \varnothing$. We define

\begin{equation*}
C(\varnothing) := 1
\end{equation*}

\noindent By Lemma \ref{onepointed}, we have

\begin{equation*}
L_{\varnothing}(g) = 24^gg!\left<\tau_{3g - 2}\tau_0^n\right>_g = 1
\end{equation*}

\noindent The proof of Theorem \ref{MainTheorem} consists of three parts. We begin by showing that

\begin{equation*}
L_{\vec{d}}(g) := 24^gg!C(\vec{d})\left<\tau_{\vec{d}}\tau_{3g - 2 + n - |\vec{d}|}\right>_g
\end{equation*}

\noindent is an integer-valued polynomial in $g$ whose leading coefficient is $6^{|\vec{d}|}$. This is proven by carefully examining the recursions that determine $L_{\vec{d}}(g)$, and making sure that the property of being an integer-valued polynomial with leading coefficient $6^{|\vec{d}|}$ is preserved under these recursions. \\

Once this is established, we can consider the shifted polynomial

\begin{equation*}
L_{\vec{d}}(g + m(\vec{d})) = 24^{g + m(\vec{d})}(g + m(\vec{d}))!C(\vec{d})\left<\tau_{d_1}\ldots\tau_{d_n}\tau_{3(g + m(\vec{d})) - 2 + n - |\vec{d}|}\right>_{g + m(\vec{d})}
\end{equation*}

\noindent (see the explanation directly following the statement of Theorem \ref{MainTheorem} for a justification as to why $m(\vec{d})$ is a natural shift of the genera) The second part of the proof is to show that $L_{\vec{d}}(g + m)$ has a nonnegative $f^*$-vector. The strategy is the same as in the first part, that is, we make sure that the property of having a nonnegative $f^*$-vector is preserved under the recursions that determine $L_{\vec{d}}(g + m)$. \\

\noindent The final part of the proof is to apply Breuer's theorem (Theorem \ref{Breuer}), which ensures that one can always find a partial polytopal complex $P_{\vec{d}}$ whose Ehrhart polynomial is $L_{\vec{d}}(g + m)$.

\subsection{$L_{\vec{d}}(g)$ Is An Integer-Valued Polynomial}

\noindent Consider the intersection number

\begin{equation*}
L_{\vec{d}}(g) := 24^gg!C(\vec{d})\left<\tau_{d_1}\ldots\tau_{d_n}\tau_{d_{n + 1}}\right>_g
\end{equation*}

\noindent The String Equation and the Dilaton Equation implies:

\begin{Lemma}[String and Dilaton Equation for $L_{\vec{d}}(g)$]\label{stringanddilaton}

The String Equation and the Dilaton Equation, respectively, imply that

\begin{align*}
L_{\vec{d}\cup\{0\}}(g) & = \left(\sum_{i = 1}^n(2d_i + 1)L_{\vec{d}\setminus\{d_i\}\cup\{d_i - 1\}}(g)\right) + L_{\vec{d}}(g) \\
L_{\vec{d}\cup\{1\}}(g) & = (6g - 3 + 3n)L_{\vec{d}}(g)
\end{align*}

\end{Lemma}

\begin{proof}
Indeed, we have

\begin{align*}
L_{\vec{d}\cup\{0\}}(g) & = 24^gg!C(\vec{d})\left<\tau_{\vec{d}}\tau_0\tau_{3g - 2 + (n + 1) - |\vec{d}|}\right>_g \\
& = 24^gg!C(\vec{d})\left[\sum_{i = 1}^n\left<\tau_{d_1}\ldots\tau_{d_i - 1}\ldots\tau_{d_n}\tau_{3g - 2 + n - (|\vec{d}|) - 1}\right>_g + \left<\tau_{\vec{d}}\tau_{3g - 2 + n - |\vec{d}|}\right>_g\right] \\
& = \left(\sum_{i = 1}^n24^gg!(2d_i + 1)C(d_1, \ldots, d_i - 1, \ldots d_n)\left<\tau_{d_1}\ldots\tau_{d_i - 1}\ldots\tau_{d_n}\tau_{3g - 2 + n - (|\vec{d}| - 1)}\right>_g\right) \\
& + 24^gg!C(\vec{d})\left<\tau_{\vec{d}}\tau_{3g - 2 + n - |\vec{d}|}\right>_g \\
& = \left(\sum_{i = 1}^n(2d_i + 1)L_{\vec{d}\setminus\{d_i\}\cup\{d_i - 1}(g)\right) + L_{\vec{d}}(g)
\end{align*}

\noindent Furthermore,

\begin{align*}
L_{\vec{d}\cup\{1\}}(g) & = 24^gg!C(\vec{d}\cup\{1\})\left<\tau_{\vec{d}}\tau_1\tau_{3g - 2 + (n + 1) - (|\vec{d}| + 1)}\right>_g \\
& = 24^gg!C(\vec{d})3!!(2g - 2 + (n + 1))\left<\tau_{\vec{d}}\tau_{3g - 2 + n - |\vec{d}|}\right>_g \\
& = 3(2g -1+ n)L_{\vec{d}}(g) \\
& = (6g - 3 + 3n)L_{\vec{d}}(g)
\end{align*} 

\end{proof}

\noindent Now consider the case that there exists some $d_i$, say $d_1$, such that $d_1 \geq 2$. We can use the Virasoro constraints (Theorem \ref{virasoro}) to evaluate $L_{\vec{d}}(g)$ by letting $d_1$ play the role of $`m + 1'$ , so that $m = d_1 - 1$. Multiplying through by $24^gg!C(\vec{d})$, and dividing by $(2(d_1 - 1) + 3)!! = (2d_1 + 1)!!$, the Virasoro constraints tell us that

\begin{equation}\label{allterms}
L_{\vec{d}}(g) = \frac{1}{(2d_1 + 1)!!}24^gg!C(\vec{d})\left(T_1(\vec{d}) + T_2(\vec{d}) + T_3(\vec{d}) + T_4(\vec{d})\right)
\end{equation}

\noindent where

\begin{align*}
T_1(\vec{d}) & :=  \sum_{i = 2}^n\frac{(2d_i + 2(d_1 - 1) + 1)!!}{(2d_i - 1)!!}\left<\tau_{d_2}\ldots\tau_{d_i + d_1 - 1}\ldots\tau_{d_n}\tau_{d_{n + 1}}\right>_g \\
T_ 2(\vec{d}) & := \frac{(2(3g - 2 + n - |\vec{d}|) + 2(d_1 - 1) + 1)!!}{(2(3g - 2 + n - |\vec{d}|) - 1)!!}\left<\tau_{d_2}\ldots\tau_{d_n}\tau_{d_{n + 1} + d_1 - 1}\right>_g \\
T_3(\vec{d}) & := \frac{1}{2}\sum_{a + b = d_1 - 2}(2a + 1)!!(2b + 1)!!\left<\tau_a\tau_b\tau_{d_2}\ldots\tau_{d_n}\tau_{d_{n + 1}}\right>_{g - 1} \\
T_4(\vec{d}) & := \frac{1}{2}\sum_{\substack{a + b = d_1 - 2 \\ I \amalg J = \{2, \ldots, n + 1\} \\ g_1 + g_2 = g}}(2a + 1)!!(2b + 1)!!\left<\tau_a\tau_I\right>_{g_1}\left<\tau_b\tau_J\right>_{g_2}
\end{align*}

\noindent In the term $T_4(\vec{d})$, we are using the shorthand notation

\begin{equation*}
I \amalg J = \{2, \ldots, n + 1\} \implies \tau_I := \prod_{i \in I}\tau_{d_i} \ \tau_J := \prod_{i \in J}\tau_{d_i} 
\end{equation*}

\noindent Our goal now is express the right hand side of Equation \ref{allterms} in terms of $L_{\vec{\ell}}(g)$, where $|\vec{\ell}| < |\vec{d}|$. Let's begin with $T_1(\vec{d})$. First, define

\begin{equation*}
\vec{d}(i) := (d_2, d_3, \ldots, d_{i - 1}, d_i + d_1 - 1, d_{i + 1}, \ldots d_n)
\end{equation*}

\noindent Then

\begin{align*}
C(\vec{d}(i)) & = \frac{C(\vec{d})}{(2d_1 + 1)!!} \cdot \frac{(2(d_i + d_1 - 1) + 1)!!}{(2d_i + 1)!!} \\
& = \frac{C(\vec{d})}{(2d_1 + 1)!!} \cdot \frac{(2d_i + (2d_1 - 1))!!}{(2d_i + 1)!!}
\end{align*}

\noindent which implies that

\begin{equation*}
(2d_i + 1)C(\vec{d}(i)) = \frac{C(\vec{d})}{(2d_1 + 1)!!} \cdot \frac{(2d_i + (2d_1 - 1))!!}{(2d_i - 1)!!}
\end{equation*}

\noindent Therefore, we have

\begin{equation}\label{term1}
\frac{C(\vec{d})}{(2d_1 + 1)!!}24^gg!T_1(\vec{d}) = \sum_{i = 2}^n(2d_i + 1)L_{\vec{d}(i)}(g)
\end{equation}

\noindent For the term involving $T_2(\vec{d})$, we have

\begin{align*}
\frac{(2(3g - 2 + n - |\vec{d}|) + 2(d_1 - 1) + 1)!!}{(2(3g - 2 + n - |\vec{d}|) - 1)!!} & = \frac{(6g - 4 + 2n - 2|\vec{d}| + (2d_1 - 1))!!}{(6g - 4 + 2n - 2|\vec{d}| - 1)!!} \\
& = \prod_{k = 1}^{d_1}(6g - 4 + 2n - 2|\vec{d}| + (2k - 1))
\end{align*}

\noindent and therefore,

\begin{equation}\label{term2}
\frac{C(\vec{d})}{(2d_1 + 1)!!}24^gg!T_2(\vec{d}) = \left(\prod_{k = 1}^{d_1}(6g - 4 + 2n - 2|\vec{d}| + (2k - 1))\right)L_{\vec{d}\setminus\{d_1\}}(g)
\end{equation}

\noindent For the term involving $T_3(\vec{d})$,

\begin{align*}
\frac{1}{(2d_1 + 1)!!}24^gg!C(\vec{d})T_3(\vec{d}) = 24g\left(\frac{1}{2}\sum_{a + b = d_1 - 2}24^{g - 1}(g - 1)!C(\vec{d}\setminus\{d_1\}\cup\{a, b\})\left<\tau_a\tau_b\tau_{d_2}\ldots\tau_{d_n}\tau_{d_{n + 1}}\right>_{g - 1}\right)
\end{align*}

\noindent The expression $d_{n + 1} = 3g - 2 + n - |\vec{d}|$ can be written as

\begin{equation*}
3(g - 1) - 2 + (n + 1) - (|\vec{d}| - d_1 + a + b)
\end{equation*}

\noindent and therefore, 

\begin{equation}\label{term3}
\frac{1}{(2d_1 + 1)!!}24^gg!C(\vec{d})T_3(\vec{d}) = 12g\sum_{a + b = d_1 - 2}L_{\vec{d}\setminus\{d_1\}\cup\{a, b\}}(g - 1)
\end{equation}

\noindent Finally, we consider the term containing $T_4(\vec{d})$. Using the symmetry of the summation, we can get rid of the factor of $\frac{1}{2}$:

\begin{align*}
T_4 = \frac{1}{2} & \sum_{\substack{a + b = d_1 - 2 \\ I \amalg J = \{2, \ldots, n + 1\} \\ g_1 + g_2 = g}}(2a + 1)!!(2b + 1)!!\left<\tau_a\tau_I\right>_{g_1}\left<\tau_b\tau_J\right>_{g_2} \\
= \frac{1}{2} & \left(\sum_{\substack{a + b = d_1 - 2 \\ I \amalg J = \{2, \ldots, n\} \\ g_1 + g_2 = g}}(2a + 1)!!(2b + 1)!!\left<\tau_a\tau_I\tau_{d_{n + 1}}\right>_{g_1}\left<\tau_b\tau_J\right>_{g_2} \right. \\
& \left. + \sum_{\substack{a + b = d_1 - 2 \\ I \amalg J = \{2, \ldots, n\} \\ g_1 + g_2 = g}} (2a + 1)!!(2b + 1)!!\left<\tau_a\tau_I\right>_{g_1}\left<\tau_b\tau_J\tau_{d_{n + 1}}\right>_{g_2}\right) \\
= & \sum_{\substack{a + b = d_1 - 2 \\ I \amalg J = \{2, \ldots, n\} \\ g_1 + g_2 = g}}(2a + 1)!!(2b + 1)!!\left<\tau_a\tau_I\right>_{g_1}\left<\tau_b\tau_J\tau_{d_{n + 1}}\right>_{g_2}
\end{align*} 

\noindent For any partition $I \amalg J = \{2, \ldots, n\}$, and for any pair $(g_1, g_2)$ such that $g_1 + g_2 = g$, we have

\begin{align*}
\frac{C(\vec{d})}{(2d_1 + 1)!!} & = C(\vec{d}_I)C(\vec{d}_J) \\
24^g & = 24^{g_1}24^{g_2} \\
g! & = g_1!g_2!{g \choose g_1}
\end{align*}

\noindent Therefore,

\begin{align*}
\frac{C(\vec{d})}{(2d_1 + 1)!!}24^gg!T_4(\vec{d}) & = \sum_{\substack{a + b = d_1 - 2 \\ I \amalg J = \{2, \ldots, n\} \\ g_1 + g_2 = g}}(2a + 1)!!24^{g_1}g_1!C(\vec{d}_I)\left<\tau_a\tau_I\right>_{g_1}(2b + 1)!!24^{g_2}g_2!C(\vec{d}_J)\left<\tau_b\tau_J\tau_{d_{n + 1}}\right>_{g_2}{g \choose g_1} \\
& = \sum_{\substack{a + b = d_1 - 2 \\ I \amalg J = \{2, \ldots, n\} \\ g_1 + g_2 = g}}(2a + 1)!!L_{\vec{d}_I}(g_1)L_{\vec{d}_J\cup\{b\}}(g_2){g \choose g_1}
\end{align*}

\noindent Now, notice that, for any given pairs $(a, b)$ and $(I, J)$ such that $a + b = d_1 - 2$ and $I \amalg J = \{2, \ldots, n\}$, there exists \emph{at most} one pair $(g_1, g_2) = (g_1, g - g_1)$ such that $L_{\vec{d}_I}(g_1)$ and $L_{\vec{d}_J}(g_2)$ are simultaneously non-zero. In fact, it is a simple computation to determine $g_1$ (and consequently $g_2 = g - g_1$) as a function of $I$ and $a$:

\begin{equation*}
3g_1 - 3 + 1 + |I| = a + |\vec{d}_I| \implies g_1 = \frac{1}{3}(a + |\vec{d}_I| + 2 - |I|)
\end{equation*} 

\begin{Remark}
Implicit in our explanation is the assumption that $\left<\tau_a\tau_I\right>_{\frac{1}{3}(a + |\vec{d}_I| + 2 - |I|)}$ is zero whenever $\frac{1}{3}(a + |\vec{d}_I| + 2 - |I|)$ is not a positive integer. Geometrically, this is just the simple statement that we are excluding the possibility of fractional genera. 
\end{Remark}

\noindent This is all to say that

\begin{equation}\label{term4}
\frac{C(\vec{d})}{(2d_1 + 1)!!}24^gg!T_4(\vec{d}) = \sum_{\substack{a + b = d_1 - 2 \\ I \amalg J = \{2, \ldots, n\}}} (2a + 1)!!L_{\vec{d}_I}\left(g_1(I, a)\right)L_{\vec{d}_J}\left(g - g_1(I, a)\right){g \choose g_1(I, a)}
\end{equation}

\noindent where $g_1(I, a) := \frac{1}{3}(a + |\vec{d}_I| + 2 - |I|)$. Putting all of these calculations together, we obtain the following recursion:

\begin{Proposition}[Virasoro Constraints for $L_{\vec{d}}(g)$]\label{virasoroforL}
Let $\vec{d} := (d_1, \ldots, d_n) \in \mathbb{Z}^n_{\geq 0}$ where $d_1 \geq 2$. We have

\begin{align*}
L_{\vec{d}}(g) & = \sum_{i = 2}^n(2d_i + 1)L_{\vec{d}(i)}(g) + \left(\prod_{k = 1}^{d_1}(6g - 4 + 2n - 2|\vec{d}| + (2k - 1))\right)L_{\vec{d}\setminus\{d_1\}}(g) \\
& + 12g\sum_{a + b = d_1 - 2}L_{\vec{d}\setminus\{d_1\}\cup\{a, b\}}(g - 1) \\
& + \sum_{\substack{a + b = d_1 - 2 \\ I \amalg J = \{2, \ldots, n\}}} (2a + 1)!!L_{\vec{d}_I}(g_1)L_{\vec{d}_J\cup\{b\}}(g - g_1){g \choose g_1}
\end{align*}

\noindent where, in the last summation on the right hand side, 

\begin{align*}
g_1 & := \frac{1}{3}\left(a + |\vec{d}_I| + 2 - |I|\right) \\
L_{\vec{d}_I}(g_1) & := \begin{cases}
L_{\vec{d}_I}\left(\frac{1}{3}(a + |\vec{d}_I| + 2 - |I|)\right), \quad a + |\vec{d}_I| + 2 - |I| \equiv 0 \ \text{mod} \ 3 \\
0, \quad \text{otherwise}
\end{cases}
\end{align*}

\end{Proposition}

\begin{proof}
This is the combination of Equations \ref{allterms}, \ref{term1}, \ref{term2}, \ref{term3}, \ref{term4}.
\end{proof}

\noindent We have the following direct corollary of Proposition \ref{virasoroforL}:

\begin{Corollary}\label{integervalued}
The intersection number $L_{\vec{d}}(g)$ is an integer-valued polynomial in $g$ of degree $|\vec{d}|$ with leading coefficient $6^{|\vec{d}|}$.
\end{Corollary}

\begin{proof}
We first establish a few base cases. By the Dilaton Equation, we have

\begin{align*}
L_{(1)}(g) & = 24^gg!C((1))\left<\tau_1\tau_{3g - 2 + 1 - 1}\right>_g \\
& = 24^gg!(3!!)\left<\tau_1\tau_{3g - 2}\right>_g \\
& = (2g - 2 + 1)24^gg!(3!!)\left<\tau_{3g - 2}\right>_g \\
& = 6g - 3
\end{align*}

\begin{align*}
L_{(1, 1)}(g) & = (6g - 3 + 3)L_{(1)}(g) \\
& = (6g)(6g - 3) 
\end{align*}

\noindent Using Proposition \ref{virasoroforL}, we have

\begin{align*}
L_{(2)}(g) & = (6g - 4 + 2 - 4 + 1)(6g - 4 + 2 - 4 + 3) + 12g \\
& = (6g - 5)(6g - 3) + 12g \\
& = 36g^2 - 48g + 15 + 12g \\
& = 36g^2 - 36g + 15
\end{align*}

\noindent Thus, if $|\vec{d}| \leq 2$, the desired result is satisfied. \\

\noindent By way of induction, suppose there exists an integer $d > 1$ such that $L_{\vec{\ell}}(g)$ is an integer valued polynomial of degree $|\vec{\ell}|$ with leading coefficient $6^{|\vec{\ell}|}$ whenever $|\vec{\ell}| < d$. Let $\vec{d} \in \mathbb{Z}^n_{\geq 0}$ be an integer vector such that $|\vec{d}| = d$. Without loss of generality, by Lemma \ref{stringanddilaton}, assume $d_i \geq 2$ for all $i$. By Proposition \ref{virasoroforL}, $L_{\vec{d}}(g)$ is a sum of four terms. What we need to check is that all four terms sum to an integer-valued polynomial of degree $d$. The first term is

\begin{equation*}
\sum_{i = 2}^n(2d_i + 1)L_{\vec{d}(i)}(g)
\end{equation*}

\noindent Since $|\vec{d}(i)| = |\vec{d}| - d_1 - d_i + (d_i + d_1 - 1) = d - 1 < d$, by the induction hypothesis, this term is an integer-valued polynomial in $g$ of degree $|\vec{d}(i)| = |\vec{d}| - 1$ \\

\noindent The second term is 

\begin{equation*}
\left(\prod_{k = 1}^{d_1}(6g - 4 + 2n - 2|\vec{d}| + (2k - 1))\right)L_{\vec{d}\setminus\{d_1\}}(g)
\end{equation*}

\noindent By the induction hypothesis, this is an integer-valued polynomial in $g$ of degree $|\vec{d}\setminus\{d_1\}| + d_1 = d$. \\

\noindent The third term is

\begin{equation*}
12g\sum_{a + b = d_1 - 2}L_{\vec{d}\setminus\{d_1\}\cup\{a, b\}}(g - 1)
\end{equation*}

\noindent By the induction hypothesis, this is an integer-valued polynomial in $g$ of degree $1 + |\vec{d}\setminus\{d_1\}\cup\{a, b\}| = 1 + d - d_1 + d_1 - 2 = d - 1 < d$ \\

\noindent Finally, the fourth term is 

\begin{equation*}
\sum_{\substack{a + b = d_1 - 2 \\ I \amalg J = \{2, \ldots, n\}}} (2a + 1)!!L_{\vec{d}_I}(g_1)L_{\vec{d}_J\cup\{b\}}(g - g_1){g \choose g_1}
\end{equation*}

\noindent The binomial term in the sum, ${g \choose g_1}$, is an integer-valued polynomial in $g$ of degree $g_1 = \frac{1}{3}(a + |\vec{d}_I| + 2 - |I|)$. Therefore, by the induction hypothesis, each summand is an integer-valued polynomial in $g$ of degree

\begin{align*}
g_1 + |\vec{d}_J| + b & = \frac{1}{3}(a + |\vec{d}_I| + 2 - |I|) + |\vec{d}_J| + b \\
& = \left(\frac{a}{3} + b\right) + \left(\frac{|\vec{d}_I|}{3} + |\vec{d}_J|\right) + \left(\frac{2}{3} - \frac{|I|}{3}\right) \\
& \leq \left(a + b\right) + \left(|\vec{d}_I| + |\vec{d}_J|\right) + \frac{2}{3} \\
& = d_1 - 2 + |\vec{d}| - d_1 + \frac{2}{3} \\
& = |\vec{d}| - \frac{4}{3} \\
& < d
\end{align*} 

\noindent Thus, it follows that $L_{\vec{d}}(g)$ is an integer-valued polynomial of degree $|\vec{d}| = d$. Its leading coefficient comes from the contribution of the fourth term, which is

\begin{equation*}
6^{d_1}6^{|\vec{d}\setminus \{d_1\}|} = 6^{|\vec{d}|}
\end{equation*}

\noindent as desired. 
\end{proof}

\subsection{The $f^*$-Vector of $L_{\vec{d}}(g + m(\vec{d}))$ Is Nonnegative}

\noindent Now we need a second corollary that says the $f^*$-vector of $L_{\vec{d}}(g + m)$ is nonnegative. However, before we prove that corollary, we need the following lemma:

\begin{Lemma}\label{secondterm}
Let $\vec{d} = (d_1, \ldots, d_n) \in \mathbb{Z}^n_{\geq 2}$ be an integer vector, and define $m = m(\vec{d}) := \left\lceil \frac{2 - n + |\vec{d}|}{3} \right\rceil - 1$. Then the polynomial

\begin{equation*}
\prod_{k = 1}^{d_1}(6(g + m) - 4 + 2n - 2|\vec{d}| + (2k - 1))
\end{equation*}

\noindent has nonnegative $f^*$-vector.

\end{Lemma}

\begin{proof}
For $1\leq k \leq d_1$, consider the linear polynomial

\begin{equation*}
L_k(g) := 6(g + m) - 4 + 2n - 2|\vec{d}| + (2k - 1)
\end{equation*}

\noindent Suppose we compute the $f^*$-vector of this linear polynomial, i.e. we find integers $f_0^*$ and $f_1^*$ such that

\begin{equation*}
L_k(g) = f_0^*{g - 1 \choose 0} + f_1^*{g - 1 \choose 1}
\end{equation*}

\noindent This means $f_0^* = L_k(1)$ and $f_1^* = \text{leading coefficient of $L_k(g)$} = 6$. So all we need to check is whether $f_0^* = L_k(1)$ is always nonnegative. \\

\noindent In the case that $2 - n + |\vec{d}| \equiv 0 \ \text{mod} \ 3 \implies \left\lceil \frac{2 - n + |\vec{d}|}{3} \right\rceil = \frac{2 - n + |\vec{d}|}{3}$, we have

\begin{align*}
f_0^* & = L_k(1) \\
& = 6\left(1 + \frac{2 - n + |\vec{d}|}{3} - 1\right) - 4 + 2n - 2|\vec{d}| + (2k - 1) \\
& = 2k - 1
\end{align*}

\noindent Similar calculations show that

\begin{align*}
& 2 - n + |\vec{d}| \equiv 1 \ \text{mod} \ 3 \implies \left\lceil \frac{2 - n + |\vec{d}|}{3} \right\rceil = \frac{2 - n + |\vec{d}|}{3} + \frac{2}{3} \implies f_0^* = L_k(1) = 2k + 3 \\
& 2 - n + |\vec{d}| \equiv 2 \ \text{mod} \ 3 \implies \left\lceil \frac{2 - n + |\vec{d}|}{3} \right\rceil = \frac{2 - n + |\vec{d}|}{3} + \frac{1}{3} \implies f_0^* = L_k(1) = 2k + 1
\end{align*}

\noindent Thus, the $f^*$-vector of $L_k(g)$ is nonnegative, and by Lemma \ref{product}, the product $\prod_{k = 1}^{d_1}L_k(g)$ also has nonnegative $f^*$-vector, as desired.

\end{proof}

\begin{Corollary}\label{fstarnonnegative}
For any vector $\vec{d} \in \mathbb{Z}_{\geq 0}^n$, define $m = m(\vec{d}) := \left\lceil \frac{2 - n + |\vec{d}|}{3} \right\rceil - 1$. Then the $f^*$-vector of $L_{\vec{d}}(g + m)$ is nonnegative.
\end{Corollary}

\begin{proof}
For the base cases, we have

\begin{align*}
& m((0)) = 0, \  L_{(0)}(g + 0) = 1 = 1{g - 1 \choose 0} \\
& m((1)) = 0, \ L_{(1)}(g + 0) = 6g - 3 = 3{g - 1 \choose 0} + 6{g - 1 \choose 1} \\
& m((1, 1)) = 0, \ L_{(1, 1)}(g + 0) =  18{g - 1 \choose 0} + 90{g - 1 \choose 1} + 72{g - 1 \choose 2} \\
& m((2)) = 0, \ L_{(2)} = 15{g - 1 \choose 0} + 72{g - 1 \choose 1} + 72{g - 1 \choose 2}
\end{align*}

\noindent By way of induction, suppose that $L_{\vec{\ell}}(g + m(\vec{\ell}))$ has a nonnegative $f^*$-vector for all vectors $\vec{\ell}$ where $|\vec{\ell}| < d$ for some positive integer $d > 1$. Let $\vec{d}$ be an integer vector such that $|\vec{d}| = d$. By Lemma \ref{stringanddilaton}, without loss of generality, assume $d_i \geq 2$ for all $i$. This means we can use Proposition \ref{virasoroforL} to compute $L_{\vec{d}}(g + m(\vec{d}))$. All that remains is to check that all four terms that arise in Proposition \ref{virasoroforL} add up to a polynomial with a nonnegative $f^*$vector. \\

\noindent When we use Proposition \ref{virasoroforL} to compute $L_{\vec{d}}(g + m(\vec{d}))$, the contribution coming from the first term is

\begin{equation*}
\sum_{i = 2}^n(2d_i + 1)L_{\vec{d}(i)}(g + m(\vec{d}))
\end{equation*}

\noindent Since $|\vec{d}(i)| = |\vec{d}| - 1$, by the induction hypothesis, the integer-valued polynomial

\begin{equation*}
L_{\vec{d}(i)}(g + m(\vec{d}(i)))
\end{equation*}

\noindent has a non-negative $f^*$-vector. However, notice that

\begin{align*}
m(\vec{d}(i)) & = \left\lceil \frac{2 - (n - 1) + (|\vec{d}| - 1)}{3} \right\rceil - 1 \\
& = \left\lceil \frac{2 - n + |\vec{d}|}{3} \right\rceil - 1 \\
& = m(\vec{d})
\end{align*}

\noindent Therefore, $L_{\vec{d}(i)}(g + m(\vec{d})) = L_{\vec{d}(i)}(g + m(\vec{d}(i)))$, so the contribution coming from the first term has nonnegative $f^*$-vector. \\

\noindent Now consider the contribution coming from the second term,

\begin{equation*}
\left(\prod_{k = 1}^{d_1}(6(g + m(\vec{d}) - 4 + 2n - 2|\vec{d}| + (2k - 1)\right)L_{\vec{d}\setminus\{d_1\}}(g + m(\vec{d}))
\end{equation*}

\noindent By Lemma \ref{secondterm}, the product of linear terms is an integer-valued polynomial with non-negative $f^*$-vector. Furthermore, since

\begin{align*}
m(\vec{d}\setminus\{d_1\}) & = \left\lceil \frac{2 - (n - 1) + (|\vec{d}| - d_1)}{3} \right\rceil - 1 \\
& = \left\lceil \frac{2 - n + |\vec{d}| - (d_1 - 1)}{3} \right\rceil - 1 \\
& \leq \left\lceil \frac{2 - n + |\vec{d}|}{3} \right\rceil - 1 \\
& = m(\vec{d})
\end{align*}

\noindent then by the induction hypothesis and Lemma \ref{shift}, it follows that $L_{\vec{d}\setminus\{d_1\}}(g + m(\vec{d}))$ is an integer-valued polynomial with nonnegative $f^*$-vector. Therefore, by Lemma \ref{product}, the contribution of the second term is an integer-valued polynomial with nonnegative $f^*$-vector. \\

\noindent Now consider the contribution coming from the third term,

\begin{equation*}
12(g + m(\vec{d}))\sum_{a + b = d_1 - 2}L_{\vec{d}\setminus\{d_1\}\cup\{a, b\}}(g + m(\vec{d}) - 1)
\end{equation*}

\noindent Since $|\vec{d}\setminus\{d_1\}\cup\{a, b\}| = |\vec{d}| - d_1 + (d_1 - 2) = |\vec{d}| - 2$, by the induction hypothesis,

\begin{equation*}
L_{\vec{d}\setminus\{d_1\}\cup\{a, b\}}(g + m(\vec{d}\setminus\{d_1\}\cup\{a, b\}))
\end{equation*}

\noindent is an integer-valued polynomial with nonnegative $f^*$-vector. However, notice that

\begin{align*}
m(\vec{d}\setminus\{d_1\}\cup\{a, b\}) & = \left\lceil \frac{2 - (n + 1) + |\vec{d}| - 2}{3} \right\rceil - 1 \\
& = \left\lceil \frac{2 - n + |\vec{d}| - 3}{3} \right\rceil - 1 \\
& < \left\lceil \frac{2 - n + |\vec{d}|}{3} \right\rceil - 1 \\
& = m(\vec{d})
\end{align*}

\noindent Therefore, by Lemma \ref{shift}, $L_{\vec{d}\setminus\{d_1\}\cup\{a, b\}}(g + m(\vec{d}) - 1)$ has a nonnegative $f^*$-vector. Since the $f^*$-vector of the polynomial $12g$ is $(12, 12)$, using Lemma \ref{shift} again, we see that $12(g + m(\vec{d}))$ has a nonnegative $f^*$-vector. By Lemma \ref{product}, the total contribution coming from the third term is an integer-valued polynomial with nonnegative $f^*$-vector. \\

\noindent Finally, consider the contribution coming from the fourth term,

\begin{equation*}
\sum_{\substack{a + b = d_1 - 2 \\ I \amalg J = \{2, \ldots, n\}}}(2a + 1)!! L_{\vec{d}_I}(g_1)L_{\vec{d}_J\cup\{b\}}(g + m(\vec{d}) - g_1){g + m(\vec{d}) \choose g_1}
\end{equation*}

\noindent Recall that

\begin{equation*}
L_{\vec{d}_I}(g_1) = \begin{cases}
L_{\vec{d}_I}\left(\frac{1}{3}(a + |\vec{d}_I| + 2 - |I|)\right) \quad a + |\vec{d}_I| + 2 - |I| \equiv 0 \ \text{mod} \ 3 \\
0 \quad \text{otherwise}
\end{cases}
\end{equation*}

\noindent so we can assume that $a + |\vec{d}_I| + 2 - |I| \equiv 0 \ \text{mod} \ 3$. Now, notice that, for any pairs $(a, b)$ and $(I, J)$ such that $a + b = d_1 - 2$, $I \amalg J = \{2, \ldots, n\}$, we have

\begin{align*}
(2 - n + |\vec{d}|) & - (a + |\vec{d}_I| + 2 - |I|) - (2 - (|J| + 1) + |\vec{d}_J| + b) \\
 & = -1 + (|I| + |J| - n) + (|\vec{d}| - |\vec{d}_I| - |\vec{d}_J|) - (a + b) \\
 & = -1 - 1 + d_1 - (d_1 - 2) \\
 & = 0
\end{align*}

\noindent But since $a + |\vec{d}_I| + 2 - |I| \equiv 0 \ \text{mod} \ 3$, it follows that

\begin{equation*}
2 - n + |\vec{d}| \equiv 2 - (|J| + 1) + |\vec{d}_J| + b \ \text{mod} \ 3
\end{equation*}

\noindent In particular, there exists an integer $0 \leq k \leq 2$ such that

\begin{align*}
\left\lceil \frac{2 - n + |\vec{d}|}{3} \right\rceil & = \frac{2 - n + |\vec{d}|}{3} + \frac{k}{3} \\
\left\lceil \frac{2 - (|J| + 1) + |\vec{d}_J| + b}{3} \right\rceil & = \frac{2 - (|J| + 1) + |\vec{d}_J| + b}{3} + \frac{k}{3}
\end{align*}

\noindent and thus,

\begin{align*}
& (m(\vec{d}) - g_1) - m(\vec{d}_J\cup\{b\}) \\
& = \left(\left\lceil \frac{2 - n + |\vec{d}|}{3} \right\rceil - 1 - \frac{1}{3}(a + |\vec{d}_I| + 2 - |I|)\right) - \left(\left\lceil \frac{2 - (|J| + 1) + |\vec{d}_J| + b}{3} \right\rceil - 1\right) \\
& = \frac{1}{3}\left((2 - n + |\vec{d}|)  - (a + |\vec{d}_I| + 2 - |I|) - (2 - (|J| + 1) + |\vec{d}_J| + b) + k - k\right) \\
& = 0
\end{align*}

\noindent It follows that $m(\vec{d}) - g_1 = m(\vec{d}_J\cup\{b\})$, so 

\begin{equation*}
L_{\vec{d}_J\cup\{b\}}(g + m(\vec{d}) - g_1) = L_{\vec{d}_J\cup\{b\}}(g + m(\vec{d}_J\cup\{b\}))
\end{equation*}

\noindent Therefore, since $|\vec{d}_J\cup\{b\}| = |\vec{d}_J| + b \leq |\vec{d}| - d_1 + d_1 - 2 = |\vec{d}| - 2 < |\vec{d}|$, by the induction hypothesis, $L_{\vec{d}_J \cup \{b\}}(g + m(\vec{d}) - g_1)$ is an integer-valued polynomial with nonnegative $f^*$-vector.  Furthermore, since the $f^*$-vector of ${g + m(\vec{d}) \choose g_1}$ is nonnegative (by Lemma \ref{shift}), the contribution of the fourth term is a sum of products of polynomials with nonnegative $f^*$-vector. By Theorem \ref{product}, the total contribution of the fourth term is a polynomial with nonnegative $f^*$-vector.  \\

\noindent In summary, we have shown that $L_{\vec{d}}(g + m(\vec{d}))$ is a sum of integer-valued polynomials with nonnegative $f^*$-vector, and therefore $L_{\vec{d}}(g + m(\vec{d})$ has nonnegative $f^*$-vector, as desired.

\end{proof}

\subsection{Putting the Pieces Together}

\begin{proof}[Proof of Theorem \ref{MainTheorem}]
By Corollary \ref{integervalued} and Corollary \ref{fstarnonnegative}, we know that

\begin{equation*}
L_{\vec{d}}(g + m) = 24^{g + m}(g + m)!C(\vec{d})\left<\tau_{\vec{d}}\tau_{3(g + m) - 2 + n - d}\right>_{g + m}
\end{equation*}

\noindent is an integer-valued polynomial in $g$ of degree $|\vec{d}|$ with nonnegative $f^*$-vector. By Breuer's theorem (Theorem \ref{Breuer}), there exists a partial polytopal complex $P_{\vec{d}}$ of dimension $|\vec{d}|$ such that

\begin{equation*}
L_{\vec{d}}(g + m) = \#\{\text{integer lattice points in $gP_{\vec{d}}$}\}
\end{equation*}

\noindent By Theorem \ref{volume}, the volume of $P_{\vec{d}}$ is $6^{|\vec{d}|}$.

\end{proof}

\section{Examples}\label{Examples}

In this section, we compute some examples of the Ehrhart polynomials $L_{\vec{d}}(g + m)$, along with their corresponding partial polytopal complexes. When $|\vec{d}| \leq 2$, we show that the corresponding partial polytopal complexes can be presented as \emph{inside-out polytopes}, a type of partial polytopal complex first studied by Beck and Zaslavsky \cite{beck2006inside}. \\

\begin{Definition}
Let $P \subseteq \mathbb{R}^d$ be a full dimensional integral $d$-polytope, and let $\mathcal{H}$ be a \emph{hyperplane arrangement}, that is, a finite collection of hyperplanes in $\mathbb{R}^d$. An \emph{inside-out polytope} is any set of the form

\begin{equation*}
P \setminus \left(\bigcup_{H \in \mathcal{H}} H\right)
\end{equation*}
\end{Definition}

\noindent The reasoning behind the name is that, one should think of the hyperplanes as dissecting the polytope $P$ into various regions, and the various hyperplanes serve as `boundaries' that lie on the `inside' of the polytope. \\

\subsection{$|\vec{d}| = 1$}

\noindent We've already computed the Ehrhart polynomial

\begin{equation*}
L_{(1)}(g + m) = L_1(g) = 6g - 3 = 3{g - 1 \choose 0} + 6{g - 1 \choose 1} = 3\left[{g - 1 \choose 0} + 2{g - 1 \choose 1}\right]
\end{equation*}

\noindent Define $P_{(1)}$ as the inside-out polytope

\begin{equation*}
P_{(1)} = [-3, 3]\setminus\{\pm 2, \pm 3\}
\end{equation*}

\noindent which can be visualized as

\begin{center}
\begin{tikzpicture}

\draw (-4, 0) node {$P_{(1)} = $};

\draw (-3, -0.5) node {$-3$};
\draw (-2, -0.5) node {$-2$};
\draw (-1, -0.5) node {$-1$};
\draw (0, -0.5) node {$0$};
\draw (1, -0.5) node {$1$};
\draw (2, -0.5) node {$2$};
\draw (3, -0.5) node {$3$};

\draw (-1, 0) -- (1, 0);

\filldraw (-1, 0) circle (2pt);
\filldraw (0, 0) circle (2pt);
\filldraw (1, 0) circle (2pt);

\draw (-2, 0) circle (2pt);
\draw (-3, 0) circle (2pt);
\draw (2, 0) circle (2pt);
\draw (3, 0) circle (2pt);

\draw (-1.9, 0) -- (-1.1, 0);
\draw (-2.9, 0) -- (-2.1, 0);
\draw (1.1, 0) -- (1.9, 0);
\draw (2.1, 0) -- (2.9, 0);

\end{tikzpicture}
\end{center}

\noindent The inside-out polytope $P_{(1)}$ has a unimodular triangulation given by

\begin{equation*}
\mathcal{T} = \{0\} \amalg \{-1\} \amalg \{1\} \amalg (-3, -2) \amalg (-2, -1) \amalg (-1, 0) \amalg (0, 1) \amalg (1, 2) \amalg (2, 3)
\end{equation*}

\noindent The $f^*$-vector of this triangulation is $f^* = (3, 6)$, and by Theorem \ref{unimodular}, 

\begin{equation*}
L_{(1)}(g + m) = \left| gP_{(1)} \cap \mathbb{Z} \right|
\end{equation*}

\noindent Alternatively, we can consider the inside-out polytope 

\begin{equation*}
\widetilde{P_{(1)}} = [-1, 1]\setminus \{\pm 1\}
\end{equation*}

\noindent which can be visualized as

\begin{center}
\begin{tikzpicture}

\draw (-2, 0) node {$\widetilde{P_{(1)}} = $};

\draw (-0.9, 0) -- (0.9, 0);
\filldraw (0, 0) circle (2pt);
\draw (-1, 0) circle (2pt);
\draw (1, 0) circle (2pt);

\draw (-1, -0.5) node {$-1$};
\draw (0, -0.5) node {$0$};
\draw (1, -0.5) node {$1$};

\end{tikzpicture}
\end{center}

\noindent The unimodular triangulation given by $\mathcal{T} = \{0\} \amalg (-1, 0) \amalg (0, 1)$ has support $P_{(1)}$, and has $f^*$-vector $(1, 2)$. Therefore, by Theorem \ref{unimodular}

\begin{equation*}
L_{(1)}(g) = 3 \left| g\widetilde{P_{(1)}} \cap \mathbb{Z} \right|
\end{equation*}

\subsection{$|\vec{d}| = 2$}

\noindent There are two vectors to consider, $(1, 1)$ and $(2)$. We have computed the polynomial $L_{(1, 1)}(g + m) = L_{(1, 1)}(g)$ previously in the proof of Corollary \ref{fstarnonnegative},

\begin{align*}
L_{(1, 1)}(g + m) = L_{(1, 1)}(g) & = 18{g - 1 \choose 0} + 90{g - 1 \choose 1} + 72{g - 1 \choose 2} \\
& = 18\left[ {g - 1 \choose 0} + 5{g - 1 \choose 1} + 4{g - 1 \choose 2} \right]
\end{align*}

\noindent Consider the inside-out polytope $P_{(1, 1)} = \left([-3, 3] \times [-3, 3]\right)\setminus \mathcal{H}_{(1, 1)} \subset \mathbb{R}^2$, where $\mathcal{H}_{(1, 1)}$ is the hyperplane arrangement

\begin{equation*}
\mathcal{H}_{(1, 1)} := \{x_2 = 3, x_2 = 2, x_2 = 1, x_2 = 0, x_1 = 3\}
\end{equation*}

\noindent Here is a visualization of $P_{(1, 1)}$:

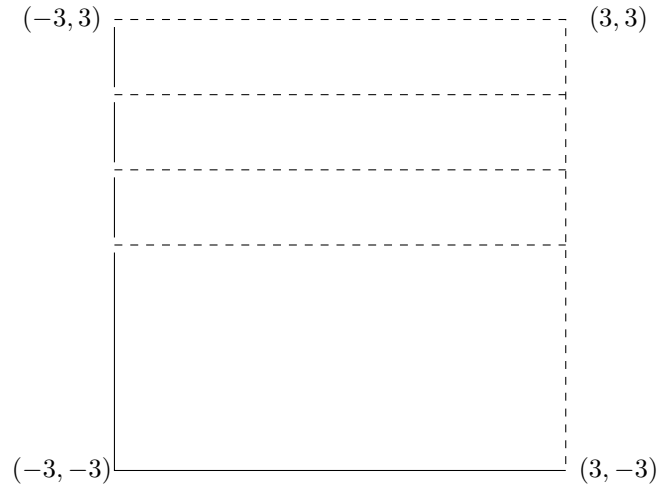
\begin{center}
\begin{tikzpicture}

\draw (-3, 2.9) -- (-3, 2.1);
\draw (-3, 1.9) -- (-3, 1.1);
\draw (-3, 0.9) -- (-3, 0.1);
\draw (-3, -0.1) -- (-3, -3) -- (3, -3);

\draw[dashed] (-3, 3) -- (3, 3) -- (3, -3);
\draw[dashed] (-3, 2) -- (3, 2);
\draw[dashed] (-3, 1) -- (3, 1);
\draw[dashed] (-3, 0) -- (3, 0);

\draw (3.7, 3) node {$(3, 3)$};
\draw (-3.7, 3) node {$(-3, 3)$};
\draw (3.7, -3) node {$(3, -3)$};
\draw (-3.7, -3) node {$(-3, -3)$};

\end{tikzpicture}
\end{center}

\noindent $P_{(1, 1)}$ admits a unimodular triangulation as suggested below:

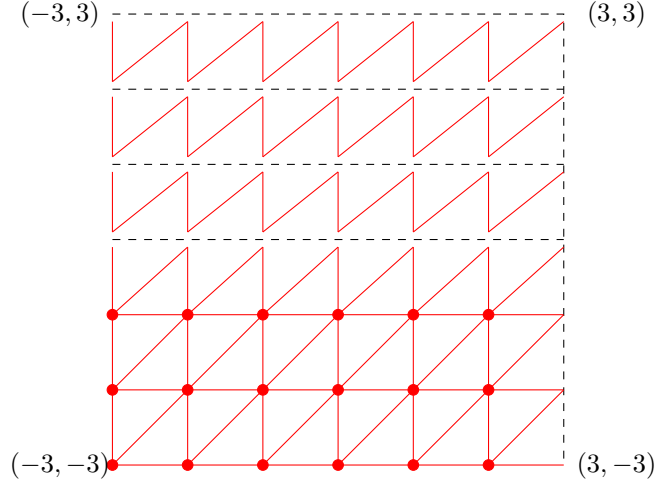
\begin{center}
\begin{tikzpicture}

\draw[red] (-3, 2.9) -- (-3, 2.1);
\draw[red] (-3, 1.9) -- (-3, 1.1);
\draw[red] (-3, 0.9) -- (-3, 0.1);
\draw[red] (-3, -0.1) -- (-3, -3) -- (3, -3);
\draw[red] (-3, -2) -- (3, -2);
\draw[red] (-3, -1) -- (3, -1);

\draw[red] (-2, 2.9) -- (-2, 2.1);
\draw[red] (-2, 1.9) -- (-2, 1.1);
\draw[red] (-2, 0.9) -- (-2, 0.1);
\draw[red] (-2, -0.1) -- (-2, -3);

\draw[red] (-1, 2.9) -- (-1, 2.1);
\draw[red] (-1, 1.9) -- (-1, 1.1);
\draw[red] (-1, 0.9) -- (-1, 0.1);
\draw[red] (-1, -0.1) -- (-1, -3);

\draw[red] (0, 2.9) -- (0, 2.1);
\draw[red] (0, 1.9) -- (0, 1.1);
\draw[red] (0, 0.9) -- (0, 0.1);
\draw[red] (0, -0.1) -- (0, -3);

\draw[red] (1, 2.9) -- (1, 2.1);
\draw[red] (1, 1.9) -- (1, 1.1);
\draw[red] (1, 0.9) -- (1, 0.1);
\draw[red] (1, -0.1) -- (1, -3);

\draw[red] (2, 2.9) -- (2, 2.1);
\draw[red] (2, 1.9) -- (2, 1.1);
\draw[red] (2, 0.9) -- (2, 0.1);
\draw[red] (2, -0.1) -- (2, -3);

\filldraw[red] (-3, -1) circle (2pt);
\filldraw[red] (-2, -1) circle (2pt);
\filldraw[red] (-1, -1) circle (2pt);
\filldraw[red] (0, -1) circle (2pt);
\filldraw[red] (1, -1) circle (2pt);
\filldraw[red] (2, -1) circle (2pt);
\filldraw[red] (-3, -2) circle (2pt);
\filldraw[red] (-2, -2) circle (2pt);
\filldraw[red] (-1, -2) circle (2pt);
\filldraw[red] (0, -2) circle (2pt);
\filldraw[red] (1, -2) circle (2pt);
\filldraw[red] (2, -2) circle (2pt);
\filldraw[red] (-3, -3) circle (2pt);
\filldraw[red] (-2, -3) circle (2pt);
\filldraw[red] (-1, -3) circle (2pt);
\filldraw[red] (0, -3) circle (2pt);
\filldraw[red] (1, -3) circle (2pt);
\filldraw[red] (2, -3) circle (2pt);

\draw[dashed] (-3, 3) -- (3, 3) -- (3, -3);
\draw[dashed] (-3, 2) -- (3, 2);
\draw[dashed] (-3, 1) -- (3, 1);
\draw[dashed] (-3, 0) -- (3, 0);

\draw (3.7, 3) node {$(3, 3)$};
\draw (-3.7, 3) node {$(-3, 3)$};
\draw (3.7, -3) node {$(3, -3)$};
\draw (-3.7, -3) node {$(-3, -3)$};

\draw[red] (-2, 2.9) -- (-3, 2.1);
\draw[red] (-2, 1.9) -- (-3, 1.1);
\draw[red] (-2, 0.9) -- (-3, 0.1);
\draw[red] (-2, -0.1) -- (-3, -1);

\draw[red] (-1, 2.9) -- (-2, 2.1);
\draw[red] (-1, 1.9) -- (-2, 1.1);
\draw[red] (-1, 0.9) -- (-2, 0.1);
\draw[red] (-1, -0.1) -- (-2, -1);

\draw[red] (0, 2.9) -- (-1, 2.1);
\draw[red] (0, 1.9) -- (-1, 1.1);
\draw[red] (0, 0.9) -- (-1, 0.1);
\draw[red] (0, -0.1) -- (-1, -1);

\draw[red] (1, 2.9) -- (0, 2.1);
\draw[red] (1, 1.9) -- (0, 1.1);
\draw[red] (1, 0.9) -- (0, 0.1);
\draw[red] (1, -0.1) -- (0, -1);

\draw[red] (2, 2.9) -- (1, 2.1);
\draw[red] (2, 1.9) -- (1, 1.1);
\draw[red] (2, 0.9) -- (1, 0.1);
\draw[red] (2, -0.1) -- (1, -1);

\draw[red] (3, 2.9) -- (2, 2.1);
\draw[red] (3, 1.9) -- (2, 1.1);
\draw[red] (3, 0.9) -- (2, 0.1);
\draw[red] (3, -0.1) -- (2, -1);

\draw[red] (-2, -1) -- (-3, -2);
\draw[red] (-1, -1) -- (-3, -3);
\draw[red] (0, -1) -- (-2, -3);
\draw[red] (1, -1) -- (-1, -3);
\draw[red] (2, -1) -- (0, -3);
\draw[red] (3, -1) -- (1, -3);
\draw[red] (3, -2) -- (2, -3);

\end{tikzpicture}
\end{center}

\noindent The $f^*$-vector of this triangulation is $f^* = (18, 90, 72)$. Therefore, by Theorem \ref{unimodular},

\begin{equation*}
L_{(1, 1)}(g + m) = \left| gP_{(1, 1)} \cap \mathbb{Z}^2 \right|
\end{equation*}

\noindent Alternatively, consider the inside-out polytope $\widetilde{P_{(1, 1)}} := \left([0, 2] \times [0, 1]\right) \setminus \widetilde{\mathcal{H}}_{(1, 1)} \subset \mathbb{R}^2$, where $\widetilde{\mathcal{H}}_{(1, 1)}$ is the hyperplane arrangement given by

\begin{equation*}
\widetilde{\mathcal{H}}_{(1, 1)} := \{x_2 = 1, x_1 = 1, x_1 = 2\}
\end{equation*} 

\noindent Here is a visualization of $\widetilde{P_{(1, 1)}}$:

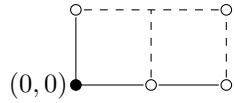
\begin{center}
\begin{tikzpicture}

\draw (0, 0) -- (0.9, 0);
\draw (1.1, 0) -- (1.9, 0);
\draw (0, 0) -- (0, 0.9);

\draw[dashed] (2, 0.1) -- (2, 0.9);
\draw[dashed] (1.9, 1) -- (0.1, 1);
\draw[dashed] (1, 0.1) -- (1, 1);

\filldraw (0, 0) circle (2pt);
\draw (1, 0) circle (2pt);
\draw (2, 0) circle (2pt);
\draw (2, 1) circle (2pt);
\draw (0, 1) circle (2pt);

\draw (-0.5, 0) node {$(0, 0)$};

\end{tikzpicture}
\end{center}

\noindent The claim is that

\begin{equation*}
L_{(1, 1)}(g) = 18 \left| g\left(P_{(1, 1)}\right) \cap \mathbb{Z}^2 \right|
\end{equation*}

\noindent Indeed, $P_{(1, 1)}$ admits a unimodular triangulation as suggested below:

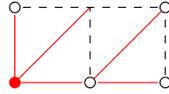
\begin{center}
\begin{tikzpicture}

\draw[color = red] (0, 0) -- (0.9, 0);
\draw[color = red] (1.1, 0) -- (1.9, 0);
\draw[color = red] (0, 0) -- (0, 0.9);
\draw[color = red] (0, 0) -- (1, 1);
\draw[color = red] (1.05, 0.05) -- (1.95, 0.95);

\draw[dashed] (2, 0.1) -- (2, 0.9);
\draw[dashed] (1.9, 1) -- (0.1, 1);
\draw[dashed] (1, 0.1) -- (1, 1);

\filldraw[color = red] (0, 0) circle (2pt);
\draw (1, 0) circle (2pt);
\draw (2, 0) circle (2pt);
\draw (2, 1) circle (2pt);
\draw (0, 1) circle (2pt);

\end{tikzpicture}
\end{center}

\noindent Since the $f^*$-vector of this triangulation is $(1, 5, 4)$, the claim follows from Theorem \ref{unimodular}. \\

\noindent Now consider the vector $\vec{d} = (2)$. We have already computed $L_{(2)}(g + m) = L_{(2)}(g)$ in the proof of Corollary \ref{fstarnonnegative}:

\begin{align*}
L_{(2)}(g + m) = L_{(2)}(g) & = 15{g - 1 \choose 0} + 72{g - 1 \choose 1} + 72{g - 1 \choose 2} \\
& = 3\left[ 5{g - 1 \choose 0} + 24{g - 1 \choose 1} + 24{g - 1 \choose 2} \right] 
\end{align*}

\noindent Consider the inside-out polytope given by

\begin{equation*}
P_{(2)} := \left( [-3, 3] \times [-3, 3] \right) \setminus \mathcal{H}_{(2)} \subset \mathbb{R}^2
\end{equation*}

\noindent where the hyperplane arrangement $\mathcal{H}_{(2)}$ is given by

\begin{equation*}
\mathcal{H}_{(2)} := \{x_1 = \pm 3, x_2 = \pm 3, x_2 = \pm 2, x_1 \pm x_2 = \pm 4, x_1 \pm x_2 = \pm 5 \} 
\end{equation*}

\noindent Here is a visualization of $P_{(2)}$:

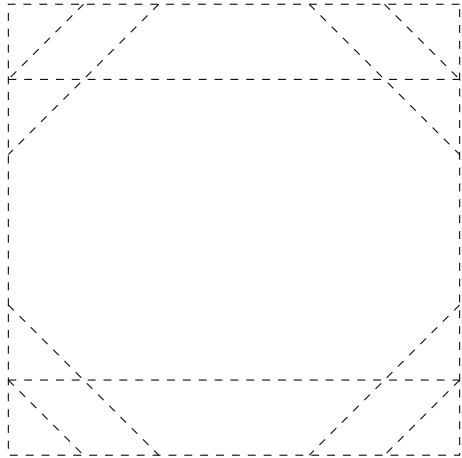
\begin{center}
\begin{tikzpicture}

\draw[dashed] (3, 3) -- (-3, 3) -- (-3, -3) -- (3, -3) -- (3, 3);
\draw[dashed] (-3, 2) -- (3, 2);
\draw[dashed] (-3, -2) -- (3, -2);
\draw[dashed] (-2, 3) -- (-3, 2);
\draw[dashed] (-1, 3) -- (-3, 1);
\draw[dashed] (2, 3) -- (3, 2);
\draw[dashed] (1, 3) -- (3, 1);
\draw[dashed] (-3, -2) -- (-2, -3);
\draw[dashed] (-3, -1) -- (-1, -3);
\draw[dashed] (3, -2) -- (2, -3);
\draw[dashed] (3, -1) -- (1, -3);

\end{tikzpicture}
\end{center}

\noindent $P_{(2)}$ admits a unimodular triangulation as suggested below:

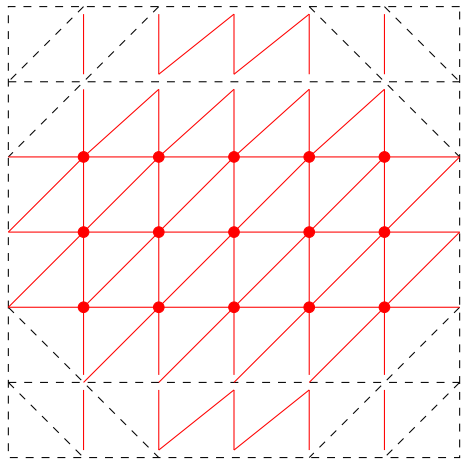
\begin{center}
\begin{tikzpicture}

\draw[dashed] (3, 3) -- (-3, 3) -- (-3, -3) -- (3, -3) -- (3, 3);
\draw[dashed] (-3, 2) -- (3, 2);
\draw[dashed] (-3, -2) -- (3, -2);
\draw[dashed] (-2, 3) -- (-3, 2);
\draw[dashed] (-1, 3) -- (-3, 1);
\draw[dashed] (2, 3) -- (3, 2);
\draw[dashed] (1, 3) -- (3, 1);
\draw[dashed] (-3, -2) -- (-2, -3);
\draw[dashed] (-3, -1) -- (-1, -3);
\draw[dashed] (3, -2) -- (2, -3);
\draw[dashed] (3, -1) -- (1, -3);

\filldraw[red] (-2, 1) circle (2pt);
\filldraw[red] (-1, 1) circle (2pt);
\filldraw[red] (0, 1) circle (2pt);
\filldraw[red] (1, 1) circle (2pt);
\filldraw[red] (2, 1) circle (2pt);

\filldraw[red] (-2, 0) circle (2pt);
\filldraw[red] (-1, 0) circle (2pt);
\filldraw[red] (0, 0) circle (2pt);
\filldraw[red] (1, 0) circle (2pt);
\filldraw[red] (2, 0) circle (2pt);

\filldraw[red] (-2, -1) circle (2pt);
\filldraw[red] (-1, -1) circle (2pt);
\filldraw[red] (0, -1) circle (2pt);
\filldraw[red] (1, -1) circle (2pt);
\filldraw[red] (2, -1) circle (2pt);

\draw[red] (-2, 2.9) -- (-2, 2.1);
\draw[red] (-2, 1.9) -- (-2, -1.9);
\draw[red] (-2, -2.1) -- (-2, -2.9);

\draw[red] (-1, 2.9) -- (-1, 2.1);
\draw[red] (-1, 1.9) -- (-1, -1.9);
\draw[red] (-1, -2.1) -- (-1, -2.9);

\draw[red] (0, 2.9) -- (0, 2.1);
\draw[red] (0, 1.9) -- (0, -1.9);
\draw[red] (0, -2.1) -- (0, -2.9);

\draw[red] (1, 2.9) -- (1, 2.1);
\draw[red] (1, 1.9) -- (1, -1.9);
\draw[red] (1, -2.1) -- (1, -2.9);

\draw[red] (2, 2.9) -- (2, 2.1);
\draw[red] (2, 1.9) -- (2, -1.9);
\draw[red] (2, -2.1) -- (2, -2.9);

\draw[red] (0, 2.9) -- (-1, 2.1);
\draw[red] (1, 2.9) -- (0, 2.1);

\draw[red] (-1, 1.9) -- (-2, 1);
\draw[red] (0, 1.9) -- (-1, 1);
\draw[red] (1, 1.9) -- (0, 1);
\draw[red] (2, 1.9) -- (1, 1);

\draw[red] (-2, 1) -- (-3, 0);
\draw[red] (-1, 1) -- (-2, 0);
\draw[red] (0, 1) -- (-1, 0);
\draw[red] (1, 1) -- (0, 0);
\draw[red] (2, 1) -- (1, 0);
\draw[red] (3, 1) -- (2, 0);

\draw[red] (-2, 0) -- (-3, -1);
\draw[red] (-1, 0) -- (-2, -1);
\draw[red] (0, 0) -- (-1, -1);
\draw[red] (1, 0) -- (0, -1);
\draw[red] (2, 0) -- (1, -1);
\draw[red] (3, 0) -- (2, -1);

\draw[red] (-1, -1) -- (-2, -2);
\draw[red] (0, -1) -- (-1, -2);
\draw[red] (1, -1) -- (0, -2);
\draw[red] (2, -1) -- (1, -2);

\draw[red] (0, -2.1) -- (-1, -2.9);
\draw[red] (1, -2.1) -- (0, -2.9);

\draw[red] (-3, 1) -- (3, 1);
\draw[red] (-3, 0) -- (3, 0);
\draw[red] (-3, -1) -- (3, -1);

\end{tikzpicture}
\end{center}

\noindent The $f^*$ vector of this triangulation is $f^* = (15, 72, 72)$, so by Theorem \ref{unimodular},

\begin{equation*}
L_{(2)}(g + m) = \left|.gP_{(2)} \cap \mathbb{Z}^2 \right|
\end{equation*}

\noindent Alternatively, consider the inside-out polytope $\widetilde{P_{(2)}} := ([-3, 3] \times [-1, 1]) \setminus \widetilde{\mathcal{H}}_{(2)}$, where $\widetilde{\mathcal{H}}_{(2)}$ is the hyperplane arrangement given by

\begin{equation*}
\widetilde{\mathcal{H}}_{(2)} = \{x_1 = \pm 1, x_2 = \pm 3, x_1 \pm x_2 = \pm 3\}
\end{equation*}

\noindent Here is a visualization of $\widetilde{P_{(2)}}$:

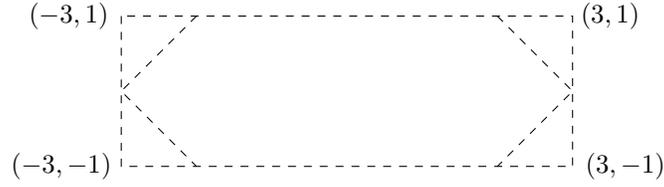
\begin{center}
\begin{tikzpicture}

\draw[dashed] (-3, 1) -- (-3, -1) -- (3, -1) -- (3, 1) -- (-3, 1);
\draw[dashed] (-2, 1) -- (-3, 0) -- (-2, -1);
\draw[dashed] (2, 1) -- (3, 0) -- (2, -1);

\draw (3.5, 1) node {$(3, 1)$};
\draw (-3.7, 1) node {$(-3, 1)$};
\draw (-3.8, -1) node {$(-3, -1)$};
\draw (3.7, -1) node {$(3, -1)$};

\end{tikzpicture}
\end{center}

\noindent The claim is that

\begin{equation*}
L_{(2)}(g + m) =  3\left| g\widetilde{P_{(2)}} \cap \mathbb{Z}^2  \right|
\end{equation*}

\noindent Indeed, $\widetilde{P_{(2)}}$ admits a unimodular triangulation as suggested below:

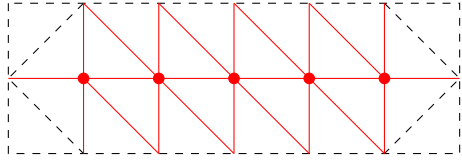
\begin{center}
\begin{tikzpicture}

\draw[dashed] (-3, 1) -- (-3, -1) -- (3, -1) -- (3, 1) -- (-3, 1);
\draw[dashed] (-2, 1) -- (-3, 0) -- (-2, -1);
\draw[dashed] (2, 1) -- (3, 0) -- (2, -1);

\filldraw[color = red] (-2, 0) circle (2pt);
\filldraw[color = red] (-1, 0) circle (2pt);
\filldraw[color = red] (0, 0) circle (2pt);
\filldraw[color = red] (1, 0) circle (2pt);
\filldraw[color = red] (2, 0) circle (2pt);

\draw[color = red] (-3, 0) -- (3, 0);
\draw[color = red] (-2, 1) -- (-2, -1);
\draw[color = red] (-1, 1) -- (-1, -1);
\draw[color = red] (0, 1) -- (0, -1);
\draw[color = red] (1, 1) -- (1, -1);
\draw[color = red] (2, 1) -- (2, -1);

\draw[color = red] (-2, 0) -- (-1, -1);
\draw[color = red] (1, 1) -- (2, 0);

\draw[color = red] (-2, 1) -- (0, -1);
\draw[color = red] (-1, 1) -- (1, -1);
\draw[color = red] (0, 1) -- (2, -1);

\end{tikzpicture}
\end{center}

\noindent The $f^*$-vector of this triangulation is $(5, 24, 24)$, so the claim follows from Theorem \ref{unimodular}. \\

\section{What Next?}\label{Future}

There are many basic open questions that seem natural in light of Theorem \ref{MainTheorem} and the computations provided in Section \ref{Examples}.  \\

\begin{Question}

\noindent Let $\vec{d} \in \mathbb{Z}^n_{\geq 0}$, and let $f^* = (f_0^*, \ldots, f_{|\vec{d}|}^*)$ be the $f^*$-vector of $L_{\vec{d}}(g + m)$. If $n(\vec{d}) := \text{gcd}(f_0^*, \ldots, f_{|\vec{d}|}^*)$, is it always possible to find an inside-out polytope $P_{\vec{d}}$ of dimension $|\vec{d}|$ such that 

\begin{equation*}
L_{\vec{d}}(g + m(\vec{d})) = n(\vec{d})L_{P_{\vec{d}}}(g) \ ?
\end{equation*}

\end{Question}

\begin{Question}

\noindent Besides $\psi$-classes, one can consider other tautological classes,

\begin{equation*}
\lambda_i, \kappa_i, \delta_{j, k} \in R^*(\mbar_{g, n})
\end{equation*}

\noindent Does an Ehrhart phenomenon still occur if we allow for the insertions of these classes as well? \\

\end{Question}

\begin{Question}

\noindent The space $\mbar_{g, n}$ can be viewed as the moduli stack of stable maps to a point. So one way to generalize is to consider the stack of stable maps to a smooth projective variety $X$,

\begin{equation*}
\mbar_{g, n}(X, d)
\end{equation*}

\noindent and try to play the same game: does an Ehrhart phenomenon still occur when considering descendent integrals on $\mbar_{g, n}(X, d)$?

\end{Question}

\bibliographystyle{alpha}
\bibliography{bibliography}

\end{document}